\documentclass{amsart}

\usepackage{amssymb,amsmath,amsthm,latexsym,booktabs,breqn}

\theoremstyle{definition}
\newtheorem{definition}{Definition}[section]
\newtheorem{notation}[definition]{Notation}

\newtheorem{remark}[definition]{Remark}

\theoremstyle{plain}
\newtheorem{lemma}[definition]{Lemma}
\newtheorem{proposition}[definition]{Proposition}
\newtheorem{theorem}[definition]{Theorem}
\newtheorem{corollary}[definition]{Corollary}

\newcommand{\rational}{\mathbb{Q}}

\newcommand{\field}{\mathbb{F}}
\newcommand{\ass}{\mathbb{A}}
\newcommand{\sltwo}{\mathfrak{sl}_2}

\begin{document}

\title[Universal associative envelopes]
{Universal associative envelopes of \\nonassociative triple systems}

\author{Hader A. Elgendy }

\address{Department of Mathematics and Statistics, University of Saskatchewan, Canada}

\email{hae431@mail.usask.ca}

\maketitle

\begin{abstract}
We construct universal associative envelopes for the nonassociative triple systems arising from
the trilinear operations of Bremner and Peresi applied to the 2-dimensional
simple associative triple system.
We use noncommutative Gr\"obner bases to determine monomial bases, structure constants,
and centers of the universal envelopes.
We show that the infinite dimensional envelopes are closely related to the down-up algebras of
Benkart and Roby.
For the finite dimensional envelopes, we determine the Wedderburn decompositions and classify the
irreducible representations.
\end{abstract}


\section{Introduction}

In this paper we use noncommutative Gr\"obner bases to construct the universal associative enveloping algebras
of the nonassociative triple systems which arise from applying the nonassociative trilinear operations classified
by Bremner and Peresi \cite{BremnerPeresi} to the 2-dimensional simple associative triple system of the first kind in the sense
of Lister \cite{Lister}, namely the space of $2 \times 2$ matrices $A = (a_{ij})$ with $a_{11} = a_{22} = 0$.

The basic goal of this work is reduce the problem of representation of these triple systems, which are the simplest examples of nonassociative triple systems satisfying polynomial identities in degree 3 and 5 to that of associative algebras. They provide natural generalizations of Lie and Jordan triple systems, and therefore have the potential for a wide range of applications to various branches of mathematics, especially to the theory of Lie groups and non-Euclidean geometry \cite{faraut, Freud, Ros}. Indeed, they have arisen in the
study of symmetric spaces \cite{Loos2}, and have been connected with the study of the Yang-Baxter equations \cite{Kamiya}.

By a multilinear $n$-ary operation we mean an element $\omega$ of the group algebra $\mathbb{Q} S_n$
of the symmetric group $S_n$ over the rational field $\mathbb{Q}$.
Following \cite{BremnerPeresi} we say that two operations $\omega_1$, $\omega_2$ are \emph{equivalent}
if they generate the same left ideal.
If $A$ is an associative algebra over $\mathbb{Q}$, then $\omega$ defines a multilinear
$n$-ary operation $\omega(a_1,\dots,a_n)$ on the underlying vector space of $A$:
  \[
  \omega = \sum_{\sigma \in S_n} x_\sigma \sigma
  \quad \implies \quad
  \omega(a_1,\dots,a_n) = \sum_{\sigma \in S_n} x_\sigma a_{\sigma(1)} \cdots a_{\sigma(n)}.
  \]
In this way we obtain a nonassociative $n$-ary algebra which we denote by $A^{\omega}$.

For $n = 2$, every bilinear operation is equivalent to either the zero operation, the associative
operation $ab$, the Lie bracket $[a,b] = ab - ba$, or the Jordan product $a \circ b = \tfrac12 (ab + ba)$.
The polynomial identities of degree $\le 3$ ($\le 4$) satisfied by the Lie bracket
(Jordan product) define Lie algebras (Jordan algebras), the most important
varieties of nonassociative algebras.
For $n = 3$, Bremner and Peresi \cite{BremnerPeresi} found canonical representatives of the equivalence classes of
trilinear operations, and identified 19 operations satisfying polynomial identities of degree 5 which do
not follow from the identities of degree 3. These operations include the Lie, anti-Lie, Jordan, and anti-Jordan triple products.

In Section \ref{twentytwo}, we find simpler operations equivalent to those of \cite{BremnerPeresi};
our operations have coefficients $\pm 1, \pm 2$ and most have coefficients $\pm 1$.
We augment this list with the symmetric, alternating and cyclic sums; see Table \ref{matrixforms}.


  \begin{table}
  \begin{tabular}{lllll}
  symmetric sum &\quad&
  $a ( b \circ c ) + b ( c \circ a ) + c ( a \circ b )$
  \\
  alternating sum   &\quad&
  $a [ b, c ] + b [ c, a ] + c [ a, b ]$
  \\
  cyclic sum   &\quad&
  $abc + bca + cab$
  \\
  Lie family $q = \infty$  &\quad&
  $[ a, [ b, c ] ]$
  \\
  Lie family $q = \frac12$ &\quad&
  $[ a, b \circ c ]$
  \\
  Jordan family $q = \infty$ &\quad&
  $abc + cba$
  \\
  Jordan family $q = 0$ &\quad&
  $( a \circ b ) c$
  \\
  Jordan family $q = \frac12$ &\quad&
  $a ( b \circ c ) + c (a \circ b ) + (c \circ a) b$
  \\
  Jordan family $q = 1$  &\quad&
  $a ( b \circ c )$
  \\
  anti-Jordan family $q = \infty$ &\quad&
  $a [ b, c ] + c [ a, b ] + [ c, a ] b$
  \\
  anti-Jordan family $q = -1$ &\quad&
  $a [ b, c ]$
  \\
  anti-Jordan family $q = \frac12$ &\quad&
  $abc - cba$
  \\
  anti-Jordan family $q = 2$ &\quad&
  $[ a, b ] c $
  \\
  fourth family $q = \infty$ &\quad&
  $abc - acb - bac$
  \\
  fourth family $q = 0$ &\quad&
  $abc - acb + bca$
  \\
  fourth family $q = 1$ &\quad&
  $abc - bac + cab$
  \\
  fourth family $q = -1$ &\quad&
  $abc + bac + cab$
  \\
  fourth family $q = 2$ &\quad&
  $abc + acb + bca$
  \\
  fourth family $q = \frac12$ &\quad&
  $abc + acb + bac$
  \\
  cyclic commutator  &\quad&
  $[ a, bc ]$
  \\
  weakly commutative  &\quad&
  $\Big\{\begin{array}{l} abc + acb + bac - cba, \\ abc - acb + 2bac \end{array}$
  \\
  weakly anticommutative &\quad&
  $\Big\{ \begin{array}{l} abc + acb - bac - cba, \\ abc + acb - 2bac \end{array}$
  \end{tabular}
  \bigskip
  \caption{The twenty-two trilinear operations}
  \label{matrixforms}
  \end{table}


In Section \ref{universal} we recall the theory of noncommutative Gr\"obner
bases for ideals in free associative algebras, and use it to define universal associative
envelopes $U(A^\omega)$ of nonassociative $n$-ary algebras $A^\omega$ defined by multilinear operations $\omega$.

In Section \ref{infinite} we recall the down-up algebras of Benkart and Roby \cite{Roby}.
We then consider the cases in which $U(A^{\omega})$ is infinite dimensional: we
determine monomial bases and structure constants, identify the center, and
determine the Gelfand-Kirillov dimension.
In every case, $U(A^{\omega})$ is either a free associative algebra, a down-up algebra,
or a quotient of a down-up algebra.

In Section \ref{finite} we consider the cases in which $U(A^{\omega})$ is finite dimensional.
We use an algorithmic version of the structure theory for finite dimensional associative algebras
to determine the Wedderburn decompositions and classify the irreducible representations.
In most cases we obtain only the trivial 1-dimensional representation and the natural 2-dimensional representation.

The results of Sections \ref{infinite} and \ref{finite} are summarized in Table \ref{summary}.
We distinguish trilinear operations of ``Lie type'' for which $U(A^{\omega})$ is infinite dimensional,
and those of ``Jordan type'' for which $U(A^{\omega})$ is finite dimensional.
Recall that for a finite dimensional Lie algebra $L$, the universal associative envelope $U(L)$ is infinite dimensional,
and the map $L \to U(L)$ is injective;
whereas for a finite dimensional Jordan algebra $J$, the universal associative envelope $U(J)$ is finite dimensional,
and the map $J \to U(J)$ is injective if and only if $J$ is special.

For the cases where the universal associative envelope is finite dimensional, as in the representation theory of Jordan algebras and triple systems \cite{{Jacobson}, {Loos}, {Mey}}, we obtain a complete description of the Wedderburn decomposition of the universal associative envelope, and this provides a complete classification of the finite dimensional irreducible representation of the triple system.  These cases also provide natural examples where the computational approach to the Wedderburn decomposition can be applied.
For the cases where the universal associative envelope is infinite dimensional, as in the representation theory of Lie algebras and triple systems \cite{Hodge}, we provide a monomial basis for the envelope and the structures constants with respect to this basis. These cases are closely related to the theory of down-up algebras, and provide a natural way to generalize down-up algebras to structures with more than two generators and/or defined by relations of degree greater than 3. These cases also provide natural generalizations of the enveloping algebras of Lie algebras, and may therefore be connected to the theory of quantum groups \cite{chari}.

We assume throughout that the base field $\field$ has characteristic 0;
if necessary, we assume that $\field$ is algebraically closed.

  \begin{table}
  \begin{tabular}{lccc}
  \multicolumn{4}{l}{Operations of Lie type}
  \\ \midrule
  & $\dim{U(A^\omega)}$ & GK-dim & $U(A^\omega)$
  \\
  $\Big\{\begin{array}{ll}\text{symmetric sum} &
 \\\text{cyclic sum} &\end{array}$&$\infty$&1&$A(-1,-1,1)/ \langle a^3, b^3\rangle$\\alternating sum&$\infty$&$\infty$&$F\langle a, b\rangle$\\$\Big\{\begin{array}{ll}\text{Lie $q = \infty$} & \\\text{anti-Jordan $q = \infty$} &\end{array}$&$\infty$&3&$A(2,-1,-2)$\\$\Big\{\begin{array}{ll}\text{Lie $q = \tfrac{1}{2}$} & \\\text{anti-Jordan $q = \tfrac{1}{2}$} &\end{array}$&$\infty$&3&$A(0,1,0)$
  \\
  \\
  \multicolumn{4}{l}{Operations of Jordan type}
  \\ \midrule
  & $\dim{U(A^\omega)}$ & GK-dim & $U(A^\omega)$
  \\
  $\left\{\begin{array}{lllllll}\text{Jordan $q = \infty$ } & \\\text{Jordan $q = \tfrac{1}{2} $} &\\\text{anti-Jordan $q = -1$}&\\\text{anti-Jordan $q = 2$}&\\\text{fourth family $q = \infty$}&\\\text{fourth family $q = -1$}&\\\text{fourth family $q = 2$}&\\\text{fourth family $q = \tfrac{1}{2}$}&\\\text{cyclic commutator }&\\\text{weakly commutative}&\\\text{weakly anticommutative}&\\\end{array}\right.$&5&0&$\rational \oplus M_{2\times 2}$\\$\Big\{\begin{array}{ll}\text{Jordan $q = 0$} & \\\text{fourth family $q = 0$} &\end{array}$&9&0&$\mathfrak{R}\oplus \rational \oplus M_{2\times 2}$\\$\Big\{\begin{array}{ll}\text{Jordan $ q = 1$ } & \\\text{fourth family $q = 1$} &\end{array}$&9&0&$\mathfrak{R} \oplus \rational \oplus M_{2\times 2}$\end{tabular}\bigskip\caption{Structure of the universal associative envelopes}\label{summary}\end{table}


\section{The twenty-two trilinear operations} \label{twentytwo}

A natural basis for $\mathbb{Q} S_3$ consists of the six permutations in lexicographical order:
$\{ abc, acb, bac, bca, cab, cba \}$.
Another natural basis consists of the matrix units
$\{ S, E_{11}, E_{12}, E_{21}, E_{22}, A \}$
for the decomposition as a direct sum
$\mathbb{Q} \oplus M_2(\mathbb{Q}) \oplus \mathbb{Q}$
of simple ideals
corresponding to the partitions $3 = 2 + 1 = 1 + 1 + 1$ which label
the irreducible representations of $S_3$.
Bremner and Peresi \cite{BremnerPeresi} give the matrix $M$ whose columns express the matrix units
as linear combinations of the permutations:
  \[
  M
  =
  \frac16
  \left[
  \begin{array}{rrrrrr}
  1 &\!\!  2 &\!\!  0 &\!\!  0 &\!\!  2 &\!\!  1 \\
  1 &\!\!  0 &\!\!  2 &\!\!  2 &\!\!  0 &\!\! -1 \\
  1 &\!\!  2 &\!\! -2 &\!\!  0 &\!\! -2 &\!\! -1 \\
  1 &\!\! -2 &\!\!  2 &\!\! -2 &\!\!  0 &\!\!  1 \\
  1 &\!\!  0 &\!\! -2 &\!\!  2 &\!\! -2 &\!\!  1 \\
  1 &\!\! -2 &\!\!  0 &\!\! -2 &\!\!  2 &\!\! -1
  \end{array}
  \right],
  \quad
  M^{-1}
  =
  \left[
  \begin{array}{rrrrrr}
  1 &\!\! 1 &\!\!  1 &\!\!  1 &\!\!  1 &\!\!  1 \\
  1 &\!\!  0 &\!\!  1 &\!\!  0 &\!\!  -1 &\!\! -1 \\
  0 &\!\!  1 &\!\! 0 &\!\!  1 &\!\! -1 &\!\! -1 \\
  0 &\!\! 1 &\!\!  -1 &\!\! -1 &\!\!  1 &\!\!  0 \\
  1 &\!\!  0 &\!\! -1 &\!\!  -1 &\!\! 0 &\!\!  1 \\
  1 &\!\! -1 &\!\!  -1 &\!\! 1 &\!\!  1 &\!\! -1
  \end{array}
  \right].
  \]
Given a trilinear operation
$\omega = x_1 abc + x_2 acb + x_3 bac + x_4 bca + x_5 cab + x_6 cba$
with coefficient vector $X = [ x_1, x_2, x_3, x_4, x_5, x_6 ]^t$, one obtains its matrix form,
  \[
  Y
  =
  \left[ \,
  y_1, \, \begin{bmatrix} y_2 & y_3 \\ y_4 & y_5 \end{bmatrix}\!, \, y_6
  \, \right],
  \]
by $Y = M^{-1} X$.
Two operations are equivalent if and only if their matrix forms are row-equivalent; hence
canonical representatives of the equivalence classes are the operations for which
each component matrix is in row canonical form.

To find the simplest representative of each equivalence class,
we first consider the $3^5 = 243$ operations whose coefficients in the permutation basis are
$[ 1, x_2, \dots, x_6 ]$ where $x_2, \dots, x_6 \in \{ 1, 0, -1 \}$.
We record the operations whose matrix forms are among the canonical representatives \cite{BremnerPeresi};
this gives 20 of the 22 operations.
We next consider the $5^5 = 3125$ operations whose coefficients in the permutation basis
satisfy $x_2, \dots, x_6 \in \{ 2, 1, 0, -1, -2 \}$; this gives the
remaining two operations, and also produces
alternative forms of the last two operations.
In more than half of the cases, the simplified operations of Table \ref{matrixforms} are more natural
than the original operations, since they can be easily expressed in terms of the Lie bracket
and the Jordan product.


\section{Gr\"obner bases and universal envelopes} \label{universal}

We first recall basic results about noncommutative Gr\"obner bases for ideals in free associative algebras.
Standard references are Bergman \cite{Bergman}, de Graaf \cite{deGraaf}; for another application to $n$-ary algebras,
see Elgendy and Bremner \cite{ElgendyBremner}.

\begin{definition}\label{dd}
Let $X = \{ x_1, \dots, x_n \}$ be a set of symbols with total order
$x_i < x_j$ $\iff$ $i < j$. The \emph{free monoid} $X^\ast$ on $X$ consists
of all (possibly empty) words $w = x_{i_1} \cdots x_{i_k}$ ($k
\ge 0$) with the associative operation of concatenation. For $w = x_{i_1}
\cdots x_{i_k} \in X^*$ the \emph{degree} is $\deg(w) = k$. The \emph{free
unital associative algebra} $F \langle X \rangle$ is the vector space  with basis $X^*$ and multiplication extended bilinearly.
The \emph{deglex} (degree-lexicographical) order on $X^\ast$ is defined by:
$u < v$ if and only if either ($i$)
$\deg(u) < \deg(v)$ or ($ii$) $\deg(u) = \deg(v)$ and $u = w x_i u'$, $v = w
x_j v'$ where $x_i < x_j$ ($w, u', v' \in X^\ast$). We say that $u \in X^*$ is
a \emph{factor} of $v \in X^*$ if there exist $w_1, w_2 \in X^*$ such that $w_1
u w_2 = v$. If $w_1$ ($w_2$) is empty then $u$ is a \emph{left}
(\emph{right}) factor of $v$.
\end{definition}

\begin{definition}
The \emph{support} of $f \in F\langle X \rangle$ is
the set of monomials $w \in X^\ast$ that occur in $f$ with nonzero
coefficient. The \emph{leading monomial} $\mathrm{LM}(f)$ is the greatest element of the support of $f$.
For an ideal $I \subseteq F\langle X \rangle$ the set of
\emph{normal words} modulo $I$ is $N(I) = \{ \, u \in X^* \mid
\text{$u \ne \mathrm{LM}(f)$ for any $f \in I$} \, \}$, and $C(I)$ is the
subspace of $F\langle X \rangle$ with basis $N(I)$.
\end{definition}

\begin{proposition} \label{C(I)proposition}
We have $F\langle X \rangle = C(I) \oplus I$.
\end{proposition}

\begin{definition}\label{No}
Let $G \subset F\langle X \rangle$ be a subset generating an ideal $I$.
An element $f \in F\langle X
\rangle$ is in \emph{normal form modulo} $G$ if no monomial in the support of $f$
has $\mathrm{LM}(g)$ as a factor for any $g \in G$.
\end{definition}

\begin{definition}
A subset $G \subset I$
is a \emph{Gr\"obner basis} of $I$ if for every $f \in I$ there exists $g \in G$
such that $\mathrm{LM}(g)$ is a factor of $\mathrm{LM}(f)$.
\end{definition}

\begin{definition}
A subset $G \subset F\langle X \rangle$ is \emph{self-reduced} if every $g
\in G$ is in normal form modulo $G \setminus \{g\}$ and every $g \in G$ is
\emph{monic}: the coefficient of $\mathrm{LM}(g)$ is 1.
\end{definition}

\begin{definition}
Consider elements $g, h \in F \langle X \rangle$ such that $\mathrm{LM}(g)$ is not a factor of $\mathrm{LM}(h)$ and
$\mathrm{LM}(h)$ is not a factor of $\mathrm{LM}(g)$.
Assume that $u, v \in X^\ast$ satisfy
$\mathrm{LM}(g)\,u = v\,\mathrm{LM}(h)$,
$u$ is a proper right factor of $\mathrm{LM}(h)$, and $v$ is a proper left factor of $\mathrm{LM}(g)$.
Then $g u - v h$ is called a
\emph{composition} of $g$ and $h$.
\end{definition}

\begin{theorem} \label{di}
If $I \subset F\langle X \rangle$ is an ideal generated by a self-reduced set
$G$, then $G$ is a Gr\"obner basis of $I$ if and only if for all compositions
$f$ of the elements of $G$ the normal form of $f$ modulo $G$ is zero.
\end{theorem}

In order to compute a Gr\"obner basis for an ideal $I \subset F\langle X \rangle$,
we start with any set $G$ of generators for $I$, and perform the following algorithm:
  \begin{enumerate}
  \item
  Compute all compositions of elements of $G$.
  Let $H$ be the set of their normal forms modulo $G$.
  \item
  Replace $G$ by $G \cup H$, and self-reduce the new set $G$ by replacing each element by its
  normal form modulo the other elements.
  \end{enumerate}
The algorithm terminates when $H = \emptyset$ or $H = \{0\}$.

We now apply this theory to the construction of universal associative envelopes.
Given a finite dimensional associative algebra $A$, and a multilinear $n$-ary operation $\omega$,
we obtain a nonassociative $n$-ary algebra $A^{\omega}$.

\begin{definition} \label{uaedefinition}
A \emph{universal associative envelope} of $A^{\omega}$ consists of a unital associative algebra
$U(A^\omega)$ and a linear map $i\colon A^{\omega} \to U(A^{\omega})$
satisfying
  \[
  i(\omega(x_1, x_2,\dots, x_n))
  =
  {\omega}
  \left(
  i(x_1), i(x_2), \dots, i(x_n)
  \right),
  \]
for all $x_1, \dots, x_n \in A^{\omega}$,
such that for any unital associative algebra $\ass$ and linear map $j\colon A^{\omega} \to
\ass$ satisfying the same equation with $j$ in place of $i$, there is a unique
homomorphism of unital associative algebras $\psi\colon U(A^\omega) \to \ass$ such that
$\psi \circ i = j$.
\end{definition}

\begin{notation} \label{ydefinition}
Let  $B = \{ e_1, e_2, \dots,e_{m}\}$ be an ordered basis of $A^\omega$,
and let $\phi\colon B \to X = \{ x_1, x_2, \dots x_m\}$ be the bijection
$\phi(e_i) = x_i$. We extend $\phi$ to a linear map, denoted by the same symbol, $\phi\colon A^{\omega}\to F \langle
X \rangle$.
\end{notation}

\begin{definition} \label{idealgenerators}
Consider the following elements of $F\langle X \rangle$:
  \[
  G_{i_1, \dots, i_n} =
  \omega   ( x_{i_1}, \dots, x_{i_n} )  -  \phi( \omega( e_{i_1}, \dots, e_{i_n} )),
  \qquad
  1\leq i_1,\dots,i_n\leq m.
  \]
Let $I$ be the ideal generated by the set of all $G_{i_1, \dots, i_n}$, and define $U(A^\omega)= F\langle X \rangle / I$.
We have the natural surjection $\pi\colon F\langle X \rangle \to U(A^\omega)$ sending $f$ to $f + I$,
and the composition $i = \pi \circ \phi\colon A^{\omega} \to U(A^\omega)$.
\end{definition}

\begin{lemma}
The algebra $U(A^{\omega})$ and the map $i\colon A^{\omega} \to U(A^\omega)$ form the universal
associative envelope of the nonassociative $n$-ary algebra $A^{\omega}$.
\end{lemma}

To obtain the elements $G_{i_1, \dots, i_n}$, we use the structure constants of $A^{\omega}$.
We then use Theorem \ref{di} to compute a Gr\"obner basis of the ideal $I$, and Proposition \ref {C(I)proposition}
to determine a basis of $U(A^{\omega})$.

\begin{notation}
We write  $\delta_{i,j}$ for the Kronecker delta, and  $\widehat{\delta}_{i,j}= 1 - \delta_{i,j}$.
\end{notation}


\section{Infinite dimensional envelopes} \label{infinite}

In this section, we consider the trilinear operations of ``Lie type''.

\begin{definition}
(Benkart and Roby \cite{Roby})
\label{doup}
Let $\field$ be a field and let $\alpha, \beta, \gamma \in \field$ be parameters.
The \emph{down-up algebra} $A(\alpha, \beta, \gamma)$ is the unital associative algebra with generators
$a, b$ and relations
  \[
  b^2 a = \alpha bab + \beta a b^2 + \gamma b,
  \qquad
  ba^2 = \alpha aba + \beta a^2 b + \gamma a.
  \]
\end{definition}

\begin{theorem}
\cite[Theorem 3.1, Corollary 3.2]{Roby}
\label{bas1}
A down-up algebra has basis
  \[
  \mathfrak{B}_1 = \{ a^i (ba)^j b^k \mid i, j, k \geq 0 \},
  \]
and its Gelfand-Kirillov dimension is $3$.
\end{theorem}

\begin{lemma}
\label{zhaolemma}
\cite[Lemma 2.2]{Zhao}
For any $c_1, c_2 \in \field$, a down-up algebra has basis
  \[
  \mathfrak{B_2} = \{ a^i(ba + c_1 ab + c_2)^j b^k \mid i, j, k \geq 0\}.
  \]
\end{lemma}

For the rest of this paper, $A$ is the associative triple system with basis
  \[
  e_1 = \begin{bmatrix} 0 & 1 \\ 0 & 0 \end{bmatrix},
  \qquad
  e_2 = \begin{bmatrix} 0 & 0 \\ 1 & 0 \end{bmatrix}.
  \]
We make the underlying vector space of $A$ into a nonassociative triple system $A^\omega$
in different ways corresponding to the trilinear operations $\omega$ of Table \ref{matrixforms}.
Let $X = \{a,b\}$ with $a < b$, and define $\phi\colon A^{\omega}\to F \langle X \rangle$ by
$\phi(e_1) = a$, $\phi(e_2) = b$.

\subsection{The symmetric sum} \label{firstthree}

The structure constants for $A^\omega$ are determined by
  \[
  [e_1,e_1,e_1] = [e_2, e_2, e_2] = 0,
  \qquad
  [ e_2, e_1, e_1] = 2e_1,
  \qquad
  [e_1, e_2, e_2] = 2e_2.
  \]

\begin{lemma}\label{smm}
A basis for $U(A^\omega)$ is the set $\{ \, a^i (ba)^j b^k \mid 0 \le i, k \le 2, j \ge 0 \, \}$.
\end{lemma}

\begin{proof}
We have $U(A^\omega) = F\langle a, b \rangle / I$ where $I$ is generated by $G = \{ G_1, G_2, G_3, G_4 \}$:
  \[
  G_1= b^3,
  \quad
  G_2= b^2a + bab + ab^2 - b,
  \quad
  G_3= ba^2 + aba + a^2b - a,
  \quad
  G_4 = a^3.
  \]
We show that the set $G$ is a Gr\"obner basis of $I$.
There are seven compositions:
  \allowdisplaybreaks
  \begin{align*}
  &S_1 = G_1 a - b G_2,\quad S_2= G_1 ba - b^2 G_2,\quad S_3 = G_1 a^2 - b^2 G_3,\quad S_4= G_2 a^2 - b^2 G_4,
  \\
  &S_5= G_2 a - b G_3,\quad S_6= G_3 a^2 - ba G_4 ,\quad S_7= G_3 a - b G_4.
  \end{align*}
We eliminate from $S_1, \dots, S_7$ all occurrences of the leading monomials of $G_1$, $G_2$, $G_3$, $G_4$;
we write $\equiv$ to indicate congruence modulo $G$:
\allowdisplaybreaks
\begin{align*}
S_1 &= - b^2 ab - ba b^2 + b^2
 \equiv - \left(-bab - ab^2 +b\right) b - bab^2+ b^2
 \equiv 0 ,
\\
S_2 &= - b^2 \left( bab + ab^2 - b\right)
 \equiv -\left(- bab - ab^2 + b \right)b^2 \equiv 0 ,
 \\
 S_3 &=  -b^2aba - b^2a^2b + b^2a = -b^2 a \left( ba +ab -1\right)
 \\
 &\equiv -\left(-bab - ab^2 + b\right) \left( ba +ab -1\right)
\\
&
 \equiv ba\left(-bab - ab^2 + b\right) + babab - bab + a \left(-bab - ab^2 + b\right) b - ab^2\\
  &\quad +bab + ab^2- b - bab  +b
 = -ba^2 b^2 + a \left(-bab - ab^2 + b\right) b
 \\
 &
 \equiv  \left(aba + ab^2 - a\right)b^2 -  a bab^2 + ab^2 \equiv 0,
\\
S_4 &= baba^2 + ab^2a^2 - ba^2
\equiv
  ba\left(-aba - a^2b \right) + a\left(-bab -ab^2 +b\right) a
  \\& \equiv -\left(- aba - a^2b + a\right) ba + a\left(-bab - ab^2 + b\right)a \equiv 0.
\end{align*}
Similar calculations show that $S_5, S_6, S_7 \equiv 0$.
Hence $G$ is a Gr\"obner basis of $I$, and Proposition \ref{C(I)proposition} completes the proof.
\end{proof}

\begin{corollary} \label{relations}
In $U(A^\omega)$, we have the relations
  \[
  a^3 = b^3 = 0, \qquad
  b^2a = - bab - ab^2 + b, \qquad
  ba^2 = -aba - a^2b + a.
  \]
Hence $U(A^\omega)$ is the quotient of $A(-1, -1,1)$ by the ideal generated by $a^3$ and $b^3$.
\end{corollary}

\begin{definition}
Consider the anti-automorphism $\zeta\colon F\langle a, b \rangle \to F\langle a, b \rangle$
defined by $\zeta(a) = b$ and $ \zeta(b) = a$.  Since
$\zeta( G_4 ) = G_1$, $\zeta( G_1 ) = G_4$, $\zeta( G_2 ) = G_3$, $\zeta( G_3 ) = G_2$,
we see that $\zeta$ induces an anti-automorphism on $U(A^\omega)$, also denoted $\zeta$.
\end{definition}

A filtration
$\{ 0 \} \subseteq V^{(0)} \subseteq V^{(1)} \subseteq \cdots \subseteq \bigcup_n V^{(n)} = U(A^\omega)$
is defined by letting $V^{(n)}$ be the subspace with basis consisting of all $a^i (ba)^j b^k$ where
$0 \le i, k \le 2$, $j \ge 0$, and $i+ 2j +k \le n$.
The associated graded algebra is
  \[
  \mathrm{gr}(U(A^\omega)) = \bigoplus_{i \geq 0}\,  \mathcal{G}^i(U(A^\omega) ),
  \qquad
  \mathcal{G}^i (U(A^\omega)) = V^{(i)}/ V^{(i -1)},
  \qquad
  V^{-1} = \{0\}.
  \]

\begin{corollary}\label{dimensionformula}
The dimension of $\mathcal{G}^n ( U(A^\omega)  )$ is $1$ if $n = 0$, $2$ if $n = 1$, $4$ if $n = 2$ or $n \ge 3$ (odd),
and $5$ if $n \ge 4$ (even).
\end{corollary}

\begin{proof}
For $n = 0$, there is one monomial: $ (i, j, k) = (0, 0, 0)$.
For $n = 1$, there are two: $(1, 0, 0)$, $ (0, 0,1)$.
For $ n = 2$, there are four: $(0, 1, 0)$, $(0, 0, 2)$, $(1, 0,1)$, $(2, 0, 0)$.
For $ n \geq 3$ and odd, four: $(0, \frac{n -1}{2}, 1)$, $(1 , \frac{n -1}{2}, 0)$,
$ (2,  \frac{n -3}{2 }, 1)$, $(1, \frac{n - 3}{2}, 2)$.
For $n \geq 4$ and even, five:  $(0, \frac{n}{2}, 0)$, $(1, \frac{n - 2}{2}, 1)$,
$(2, \frac{n - 2}{2}, 0)$, $(0, \frac{n - 2}{2}, 2)$, $ ( 2, \frac{n -4}{2}, 2)$.
\end{proof}

\begin{corollary}
The Gelfand-Kirillov dimension of $U(A^\omega)$ is $1$.
\end{corollary}

\begin{proof}
We have
\[
GK\dim U(A^\omega) = \limsup_{n\to \infty} \log_n \dim V^{(n)}
 = \lim_{n \to \infty} \frac{\ln \dim V^{(n)}}{\ln n} = 1,
\]
since Corollary \ref{dimensionformula} implies that $\dim V^{(n)}$ is
a polynomial of degree $1$.
\end{proof}

\begin{corollary}\label{gr}
A $\mathbb{Z}$-grading of $U(A^\omega)$ is given by
  \[
  U(A^\omega) =  U(A^\omega)_{-2} \oplus U(A^\omega)_{-1}\oplus U(A^\omega)_0 \oplus U(A^\omega)_1 \oplus U(A^\omega)_2,
  \]
where
$U(A^\omega)_n = \mathrm{span}\{ \, a^i (ba)^j b^k \mid j \geq 0,\,  0\leq  i,  k\leq 2,\,  i - k = n \, \}$.
\end{corollary}

\begin{proof}
Similar to \cite[Proposition 3.5]{Roby}.
\end{proof}

Our next goal is to compute the structure constants of $U(A^\omega)$.

\begin{definition}\label{sm}
For $j,\ell, r, m\geq 0 $, we define the following polynomials:
  \[
  L^m_{j, \ell ,r} = \sum^j_{t = 0}(-1)^{j+t}\binom{j}{t} a^{\ell}(ba)^{j+m-t}b^r.
  \]
\end{definition}

\begin{lemma}\label{SM}
We have
  \[
  -L^{j-1}_{m+1,\ell,r}+ L^{j-1}_{m,\ell,r} = L^j_{m, \ell, r}\,\, ( j> 0),
  \qquad
  -L^{m + 1}_{j, \ell ,r} + L^{m }_{j, \ell ,r} = L^{m}_{j+1,\ell,r}.
  \]
\end{lemma}

\begin{proof}
Use Pascal's formula for binomial coefficients.
\end{proof}

\begin{lemma}\label{a}
If $j, m \ge 0$, then in $U(A^\omega)$ we have
  \begin{align}
  \label{oone}
  (ba)^j \cdot a (ba)^m
  &=
  -\widehat{\delta}_{j,0}\, L^{j-1}_{m,2,1} +  L^m_{j,1,0},
  \\
  \label{ttwo}
  (ba)^j b \cdot (ba)^m
  &=
  -\widehat{\delta}_{m, 0}\, L^{m-1}_{j,1,2} + L^j_{m, 0, 1},
  \\
  \label{tthree}
  (ba)^j b^2 \cdot  a^2
  &=
  -L^0_{j+1,1,1}  + (ba)^{j+1} - (ba)^{j +2} \,{+}\, a^2(ba)^jb^2.
 \end{align}
\end{lemma}

\begin{proof}
For \eqref{oone}, we use induction on $j$. Clearly the claim is true for $j = 0$. To prove it for
$j = 1$, we use induction on $m$.  For $m = 0$, Lemma {\ref{relations}} implies
  \[
  (ba) a
  =
  ba^2
  =
  - aba - a^2b + a
  =
  -L^{0}_{0,2,1} + L^0_{1, 1, 0}.
  \]
By the inductive hypothesis, we have
\begin{align*}
(ba) a (ba)^m
&=
(ba) a (ba)^{m-1} ba = \left(-L^{0}_{m-1,2,1} + L^{m-1}_{1, 1, 0}\right)ba
\\
&=
(-1)^{m}\sum^{m-1}_{t = 0}(-1)^t \binom{m{-}1}{t}a^2(ba)^{m-t-1} b^2a - a(ba)^{m +1} + a(ba)^m.
\end{align*}
Use the second relation of Lemma \ref{relations}:
\begin{dmath*}
 (ba) a (ba)^m
    =  (-1)^{m}\sum^{m-1}_{t = 0}(-1)^{t+1} \binom{m{-}1}{t}a^2(ba)^{m-t}b
+(-1)^{m}\sum^{m-1}_{t = 0}(-1)^{t+1} \binom{m{-}1}{t}a^2(ba)^{m-t-1}ab^2
+ (-1)^{m}\sum^{m-1}_{t = 0}(-1)^t \binom{m{-}1}{t}a^2(ba)^{m-t-1}b
- a(ba)^{m +1} + a (ba)^m.
\end{dmath*}
Use Pascal's formula in the first sum and change index in the third sum:
  \allowdisplaybreaks
\begin{dmath*}
(ba)a (ba)^m =  (-1)^{m+1}\bigg( a^2(ba)^mb + \sum^{m-1}_{t = 1}(-1)^{t}
\binom{m}{t}a^2(ba)^{m-t}b\bigg) +(-1)^{m} \sum^{m-1}_{t = 1}(-1)^{t}
\binom{m{-}1}{t{-}1}a^2(ba)^{m-t}b
+(-1)^{m}\sum^{m-1}_{t = 0}(-1)^{t+1} \binom{m{-}1}{t}a^2(ba)^{m-t-1}ab^2
+ (-1)^{m}\sum^{m}_{t = 1}(-1)^{t-1} \binom{m{-}1}{t{-}1}a^2(ba)^{m-t}b
- a(ba)^{m +1} + a(ba)^m.
\end{dmath*}
The first $m{-}1$ terms of the second and fourth sums cancel:
  \allowdisplaybreaks
  \begin{dmath*}
  (ba) a (ba)^m =  -L^{0}_{m,2,1}+ (-1)^{m}\sum^{m-1}_{t = 0}(-1)^{t+1} \binom{m{-}1}{t}a^2(ba)^{m-t-1}ab^2 + L^m_{1, 1, 0}.
  \end{dmath*}
To complete the proof for $j = 1$, it suffices to show that the second sum is 0.
But this holds since $a^2 (ba)^\ell a = 0$ for $\ell \geq 0$.
To show this, we use induction on $\ell$; the claim is true for $\ell = 0$ by the first equation of Lemma \ref{relations}. For $\ell = 1$, Lemma \ref{relations} implies
$a^2(ba) a = a^2(-aba-a^2b+a) = 0$. For $\ell \geq 1$, Lemma \ref{relations} and the inductive hypothesis give
$
a^2(ba)^\ell a=
a^2(ba)^{\ell -1} (-aba - a^2b+a ) =0.
$
We now consider the case $j \ge 1$.
Using the inductive hypothesis and Lemma \ref{relations}, we obtain
\allowdisplaybreaks
\begin{align*}
&
(ba)^{j+1} a (ba)^m
=
{ba (ba)^j a (ba)^m = ba\big(- L^{j-1}_{m,2,1} + L^m_{j, 1, 0}\big) }
\\
&=
(-1)^j \sum^{j}_{t = 0}(-1)^t \binom{j}{t} (ba)a(ba)^{j+m-t}
\\
&=
(-1)^{m+1} \sum^j_{t = 0} \sum^{j{+}m{-}t}_{s =
0}(-1)^s\binom{j}{t}\binom{j{+}m{-}t}{s}a^2(ba)^{j+m-t-s}b
\\
&\quad
+(-1)^{j+1}\sum^j_{t = 0}(-1)^t\binom{j}{t}a(ba)^{j+m-t+1}
+(-1)^j\sum^j_{t = 0} (-1)^t\binom{j}{t} a(ba)^{j+m-t}.
\end{align*}
Use Pascal's formula in the second sum and change index in the third sum:
\begin{dmath*}
(ba)^{j+1} a (ba)^m
 =  (-1)^{m+1}\sum^j_{t = 0}\sum^{j+m-t}_{s =
0}(-1)^s\binom{j}{t}\binom{j{+}m{-}t}{s}a^2(ba)^{j+m-t-s}b
+ (-1)^{j+1}\sum^j_{t = 0}(-1)^t\binom{j{+}1}{t}a(ba)^{j+m-t+1}
-(-1)^{j+1}\sum^j_{t=1}(-1)^t\binom{j}{t{-}1}a(ba)^{j+m-t+1}
-(-1)^{j}\sum^j_{t=1}(-1)^t\binom{j}{t{-}1}a(ba)^{j+m-t+1}+a(ba)^m.
\end{dmath*}
The last two sums cancel and the previous expression simplifies to
\[
(-1)^{m+1}\sum^j_{t = 0}\sum^{j+m-t}_{s = 0}(-1)^s\binom{j}{t}\binom{j{+}m{-}t}{s}a^2(ba)^{j+m-t-s}b + L^m_{j +1, 1, 0}.
\]
To complete the proof, it suffices to show that
  \begin{dmath*}
  \sum^j_{t=0}
  \sum^{j+m-t}_{s=0}
  (-1)^s
  \binom{j}{t}
  \binom{j{+}m{-}t}{s}
  a^2(ba)^{j+m-t-s}b
  =
  \sum^m_{t=0}
  (-1)^t
  \binom{m}{t}
  a^2(ba)^{j+m-t}b.
  \end{dmath*}
Let $C_r$ denote the coefficient of $a^2(ba)^{j+m-r}b$ in the left side:
  \allowdisplaybreaks
  \begin{align*}
  C_r &=  \sum^j_{t=0}   \sum^{j+m-t}_{s=0}
  \delta_{s+t,r}
  (-1)^s
  \binom{j}{t}
  \binom{j{+}m{-}t}{s}
   =   \sum^j_{t=0}
  \binom{j}{t}
  (-1)^{r-t}
  \binom{j{+}m{-}t}{r{-}t}
  \\
  & = (-1)^r
  \sum^j_{t=0}
  (-1)^{t}
  \binom{j}{t}
  \binom{j{+}m{-}t}{r{-}t}
  = (-1)^r \binom{m}{r}.
  \end{align*}
For the last equality, see \cite[Example 10.3]{vanLintWilson}. This completes the proof of \eqref{oone}.
The proof of \eqref{ttwo} is obvious by using the anti-automorphism $\zeta$:
  \[
  (ba)^j b (ba)^m = \zeta ( (ba)^m a (ba)^j )
  =
  -\widehat{\delta}_{m,0}\, L^{m-1}_{j, 1, 2} + L^j_{m, 0, 1}.
  \]
For \eqref{tthree}, we use Lemma \ref{relations} and get
\begin{dmath}\label{al}
  {(ba)^j b^2  a^2 = (ba)^j (-bab - ab^2 +b)a
   = -(ba)^{j+2} - (ba)^j ab^2 a + (ba)^{j +1}}.
\end{dmath}
Write $ T = - (ba)^j ab^2 a$. Lemma \ref{relations} implies
\begin{dmath*}
{T = - (ba)^j a(-bab  - ab^2 +b) = (ba)^j a (ba) b + \delta_{j, 0}\, a^2 b^2 - (ba)^j ab.}
\end{dmath*}
Using \eqref{oone} and Lemma \ref{SM} we get
\allowdisplaybreaks
\begin{align*}
T
&=
\big( -\widehat{\delta}_{j,0}\, L^{j-1}_{1,2,1} + L^1_{j, 1, 0}
+ \widehat{ \delta}_{j,0}\, L^{j-1}_{0,2,1}-L^0_{j,1,0} \big) b
+ \delta_{j,0}\, a^2 b^2
\\
&= \widehat{ \delta}_{j,0}\, L^j_{0, 2, 1}\, b - L^{0}_{j+1,1,0}\, b+ \delta_{j, 0}\, a^2 b^2
= a^2 (ba)^j b^2 - L^0_{j+1,1,1}.
\end{align*}
Using $T$ in \eqref{al} completes the proof of \eqref{tthree}.
\end{proof}

\begin{theorem}\label{strc}
The structure constants of $U(A^\omega)$ are
\begin{align}
 a^i (ba)^j b^k \cdot a^\ell(ba)^m b^n
 &=
 a^i(ba)^{j+k+\ell +m} b^n, \; \text{if $(k, \ell) = (0,0)$ \text{or} $(k,\ell)=(1,1)$},\label{AA}
 \\
a^i(ba)^j \cdot a (ba)^mb^n
&=
 -\delta_{i,0} \widehat{ \delta}_{n,2}\widehat{ \delta}_{j,0}\,L^{j-1}_{m,2,n+1} + \widehat{ \delta}_{i,2}L^m_{j, i+1,n},\label{BB}
 \\
a^i(ba)^jb \cdot (ba)^mb^n
&= -\delta_{n,0} \widehat{\delta}_{i,2}\widehat{ \delta}_{m,0}\, L^{m-1}_{j,i+1,2} +\widehat{ \delta}_{n,2} \, L^j_{m, i, n+1},\label{CC}
\\
 a^i(ba)^jb^2 \cdot a(ba)^mb^n
      &=
- \delta_{n,0}\widehat{ \delta}_{i,2}\, L^m_{j, i+1, 2}
 + \widehat{ \delta}_{n,2}\, L^j_{m +1,i, n+1}, \label{DD}
\\
  a^i(ba)^jb \cdot a^2 (ba)^mb^n &= - \delta_{i,0}\widehat{ \delta}_{n,2} L^j_{m, 2,n+1} + \widehat{ \delta}_{i,2} \, L^{m}_{j +1,i + 1, n},    \label{EE}
  \\
 a^i(ba)^j b^2 \cdot (ba)^m b^n &=\delta_{m,0}\delta_{n,0}\, a^i(ba)^j b^2, \label{FF}
 \\
 a^i(ba)^j \cdot a^2 (ba)^mb^n  &= \delta_{i,0}\delta_{j,0}\,  a^2 (ba)^mb^n, \label{GG}
\end{align}
together with
\begin{align}
&
a^i (ba)^j b^2 \cdot a^2 (ba)^m b^n
\notag
\\
&
=
\sum^{j+1}_{k = 0}
(-1)^{k+j} \binom{j+1}{k}
\left[
-\delta_{n,0} \delta_{i,0}\widehat{\delta}_{m,0}\, L^{m-1}_{j-k+1,2,2} + \widehat{ \delta}_{n,2}\widehat{ \delta}_{i,2} \, L^{j - k +1}_{m, i+1, n+1}
\right]
\label{HH}
\\
&\quad
+
a^i(ba)^{j+m+1}b^n - a^i(ba)^{j+m + 2}b^n +\delta_{i,0} \delta_{m, 0}\delta_{n,0} \, a^{2}(ba)^jb^2.
\notag
\end{align}
\end{theorem}

\begin{proof}
For \eqref{AA}, use the associativity of $U(A^\omega)$.
For \eqref{BB} and \eqref{CC} use Lemma {\ref{relations}} and equations \eqref{oone} and \eqref{ttwo} of Lemma \ref{a}.
For \eqref{DD}, Lemma {\ref{relations}} implies
\begin{dmath*}
a^i (ba)^jb^2 a(ba)^mb^n = a^i(ba)^j \left(-bab - ab^2 + b\right)(ba)^m b^n
 = - a^i(ba)^{j+1}b (ba)^m b^n - \delta_{m,0}\delta_{n,0} a^i (ba)^j a b^2+ a^i(ba)^jb(ba)^m b^n.
\end{dmath*}
Using \eqref{BB} and \eqref{CC} and Lemma \ref{SM} we obtain \eqref{DD}:
\begin{dmath*}
a^i (ba)^j  b^2 a(ba)^mb^n
=  \delta_{n,0} \widehat{\delta}_{i,2}\widehat{\delta}_{m, 0} L^{m-1}_{j+1, i+1, 2} -\widehat{\delta}_{n, 2} L_{m,i, n+1}^{j+1}
  - \delta_{m,0} \delta_{n,0} \widehat{\delta}_{i, 2} L^0_{j, i+1,2}
   - \delta_{n, 0}\widehat{\delta}_{i, 2}\widehat{\delta}_{m, 0}L^{m-1}_{j, i+1, 2} +\widehat{\delta}_{n, 2}L^j_{m,i,n+1}
   = -\delta_{n, 0} \widehat{\delta}_{i, 2}\widehat{\delta}_{m, 0}\, L^m_{j, i+1, 2}
+
\widehat{\delta}_{n, 2}\, L^{j}_{m+1,i,n+1}
- \delta_{m,0} \widehat{\delta}_{n, 0} \widehat{\delta}_{i,2}\, L^0_{j, i+1,2}.
\end{dmath*}
For \eqref{EE} use \eqref{DD} and the anti-automorphism $\zeta$.
The proofs of \eqref{FF} and \eqref{GG} are obvious by Lemma \ref{relations}.
For \eqref{HH}, we use \eqref{tthree} of Lemma \ref{a} and obtain
\begin{dmath*}
a^i (ba)^j b^2a^2 (ba)^m b^n=  - a^i L^0_{j+1,1,1}(ba)^m b^n  + a^i (ba)^{j+1}(ba)^m b^n
 -a^i (ba)^{j +2} (ba)^m b^n
 {+}\,a^{i+2}(ba)^jb^2 (ba)^m b^n.
\end{dmath*}
Using Lemma \ref{relations}, we get
\begin{dmath}\label{co}
a^i (ba)^j  b^2a^2 (ba)^m b^n \notag
= \, -\widehat{\delta}_{i, 2} L^0_{j+1,i+1,1}(ba)^m b^n  + a^i(ba)^{j+1+m} b^n
   - a^i(ba)^{j+m + 2} b^n \notag
    +\delta_{i, 0} \delta_{m, 0}\delta_{n, 0}\, a^{2}(ba)^jb^2.
\end{dmath}
Write $ A = L^0_{j+1,i+1,1}(ba)^m b^n$ and use \eqref{CC} to obtain
\begin{dmath*}
A=
\sum^{j+1}_{k = 0}(-1)^{k+j + 1}\binom{j+1}{k}a^{i+1}(ba)^{j-k+1}b(ba)^m b^n
=\sum^{j+1}_{k = 0}(-1)^{k+j+ 1}\binom{j+1}{k}\left[ -\delta_{n,0} \delta_{i,0}\widehat{\delta}_{m, 0}\, L^{m-1}_{j-k+1, 2, 2}+ \widehat{\delta}_{n, 2}\widehat{\delta}_{i, 2}\,L^{j - k +1}_{m, i+1, n+1}\right].
\end{dmath*}
Using $A$ in \eqref{co} completes the proof of \eqref{HH}.
\end{proof}

Our next goal is to describe the center $Z(U(A^\omega))$ of $U(A^\omega)$.

\begin{notation}
We consider the following functions:
  \[
  \begin{array}{ccccc}
  & \gamma_1(m) & \gamma_2(m) & \gamma_3(m) & \gamma_4(m) \\
  \text{$m$ even} & m+1 & -3 & 0 & -m+2 \\
  \text{$m$ odd}  & -(m-3) & -1 & -2 & m
  \end{array}
  \]
\end{notation}

\begin{definition}
We consider the following elements:
  \begin{dmath*}
  \mathcal{Z}(m)
  =
  \widehat{\delta}_{m2}
  \sum^{m -2}_{j = 1}
  (-1)^{j+1} \binom{m{-}1}{j-1} \left(  (ba)^j- a(ba)^{j-1} b\right)+ \gamma_1(m) (ba)^{m-1}
  +
  \gamma_2(m) (ba)^m
  +
  \gamma_3(m) a(ba)^{m-1}b
  +
  \gamma_4(m) a(ba)^{m-2}b
  +
  3 a^2 (ba)^{m-2} b^2.
  \end{dmath*}
\end{definition}

\begin{theorem}
The center of $U(A^\omega)$ is the polynomial algebra in $\mathcal{Z}(m)$, $m \ge 2$:
  \[
  Z(U(A^\omega)) = \field[ \, \mathcal{Z}(m) \mid  m \geq 2 \, ].
  \]
\end{theorem}

\begin{proof}
By Corollary \ref{gr} we know $Z(U(A^\omega))$ is graded.
Thus if $z$ is central and $z = z_{-2} + z_{-1}+z_{0}+z_{1}+ z_{2}$
is its decomposition into homogenous components, then each $z_i$ is itself central.
We now show that $z\in U(A^\omega)_0$.  First assume
\[
0 \neq z_{-2}= \sum_{j \geq 0} s_j (ba)^j b^2 \in Z(U(A^\omega)), \quad s_j \in \field.
\]
It follows that
\[
0 = z_{-2}\, a - a\, z_{-2}  = \sum_{j \geq 0} s_j (ba)^j b^2a - \sum_{j \geq 0} s_j a (ba)^j b^2.
\]
Using \eqref{DD} of Theorem \ref{strc}, we see that this element is 0:
\[
\sum_{j \geq 0}\sum^j_{t = 0}(-1)^{j + t + 1} \binom{j}{t} s_j a (ba)^{j - t} b^2 -\sum_{j \geq 0}s_j (ba)^{j +1}b +\sum_{j \geq 0} s_j (ba)^jb - \sum_{j \geq 0} s_j a (ba)^j b^2.
\]
Comparing the coefficients on both sides gives $ s_j = 0$ for all $j$.
Now assume \[ 0 \neq z_2 = \sum_{j\geq 0} s_j a^2(ba)^j    \in Z(U(A^\omega)),
\quad s_j \in \field.\]
It follows that
  \[
  0  = b\, z_2 - z_2\, b = \sum_{j\geq 0} s_j b a^2(ba)^j - \sum_{j\geq 0} s_j a^2(ba)^jb.
  \]
Applying the anti-automorphism $\zeta$ to both sides gives
\allowdisplaybreaks \begin{align*}
 &\sum_{j\geq 0} s_j (ba)^j b^2 a = \sum_{j\geq 0} s_ja(ba)^j b^2.
 \end{align*}
Hence $[ \, \sum_{j \geq 0} s_j (ba)^j b^2, \, a \, ] = 0$, contradicting the previous case.
Next assume
\[0\neq z_1 = \sum_{j\geq 0} s_j\, a (ba)^j + \sum_{\ell\geq 0} t_{\ell}\, a^2(ba)^{\ell} b \in Z(U(A^\omega)),
\quad s_j, t_{\ell} \in \field.
\]
It follows that
\[
0  = b z_1 - z_1 b = \sum_{j \geq 0} s_j \, (ba)^{j +1} + \sum_{\ell\geq 0} t_{\ell}\, b a^2(ba)^{\ell} b -  \sum_{j\geq 0} s_j\, a (ba)^j b- \sum_{\ell\geq 0} t_{\ell}\, a^2(ba)^{\ell} b^2.
\]
Using \eqref{EE} of Theorem \ref{strc} gives
\begin{dmath*}
0  = \sum_{j\geq 0} s_j\,  (ba)^{j +1} - \sum_{\ell\geq 0}\sum^{\ell}_{t = 0}(-1)^{\ell +t}\binom{\ell}{t} t_{\ell}\, a^2(ba)^{\ell - t}b^2
-\sum_{\ell\geq 0}  t_{\ell}\, a(ba)^{\ell+1} b + \sum_{\ell \geq 0} t_{\ell}\, a(ba)^{\ell}b
 -  \sum_{j \geq 0} s_j\, a (ba)^j b- \sum_{\ell\geq 0} t_{\ell}\, a^2(ba)^{\ell} b^2.
\end{dmath*}
Comparing the coefficients on both sides gives $ s_j= 0 = t_{\ell}$ for all $j, \ell$. Similarly we can show that $Z(U(A^\omega))\cap U(A^\omega)_{-1} = 0$. Therefore $z\in U(A^\omega)_{0}$.
Now $U(A^\omega)_{0}$ is a commutative subalgebra: any element in $U(A^\omega)_0$ is a linear combination of
$(ba)^j$, $a (ba)^k b$, $a^2(ba)^\ell b^2$ for $j, k, \ell \geq 0$,
and these elements commute by Theorem \ref{strc}.
Using the anti-automorphism $\zeta$, we see that an element in $U(A^\omega)_0$ commutes with $a$ if and only if
it commutes with $b$.
So it suffices to determine the elements that commute with $a$. Without loss of generality, we choose
\[ z = \sum^m_{j = 0}\sum^2_{i=0} s_{i, j}\, a^i (ba)^j b^i \in U(A^\omega)_0, \quad
m\geq 0,\,\, s_{i,j} \in \field.\]
By the relations of Lemma \ref{relations} we have
\allowdisplaybreaks \begin{align}\label{ax}
az = \sum^m_{j = 0} s_{0, j}\, a(ba)^j + \sum^m_{j = 0} s_{1,j}\, a^2(ba)^jb.
\end{align}
On the other hand,
\begin{dmath*}
za = \sum^m_{j = 0} s_{0,j}\,(ba)^ja +\sum^m_{j = 0} s_{1,j}\, a(ba)^{j+1}+\sum^m_{j = 0}s_{2,j}\, a^2(ba)^j b^2a.
\end{dmath*}
Using \eqref{BB} and \eqref{DD} of Theorem \ref{strc} we obtain
\begin{dmath} \label{xaa}
za =  \, s_{0, 0}\, a + \sum^m_{j = 1}s_{0, j}\,\bigg(-a^2(ba)^{j-1}b +\sum^j_{t = 0}(-1)^{j+t} \binom{j}{t} a(ba)^{j-t}\bigg)\notag
+\sum^m_{j = 0} s_{1, j}\, a(ba)^{j+1}- \sum^m_{j = 0}s_{2, j}\, a^2(ba)^{j+1}b + \sum^m_{j=0}s_{2, j}\,a^2(ba)^j b.
\end{dmath}
We write
  \[
  A = \sum^m_{j = 1} \sum^j_{t = 0}(-1)^{j+t} \binom{j}{t} s_{0, j}\, a(ba)^{j-t},
  \qquad
  E=\sum^m_{j = 0} s_{1, j}\, a(ba)^{j+1}.
  \]
We obtain
\begin{align}
\label{AE}
&
A + E = \sum^m_{r = 0} \bigg( \sum^m_{j = 1} (-1)^r \binom{j}{r} s_{0,j} \bigg) a(ba)^r + \sum^{m+1}_{j = 1}
 s_{1,j-1}\,a(ba)^{j}
\\
&=
\sum^m_{j = 1} s_{0, j}\, a +  \sum^m_{r = 1}\bigg( \sum^m_{j = 1} (-1)^r \binom{j}{r}s_{0, j} + s_{1,r-1}\,\bigg)a(ba)^r
+ s_{1, m}\, a(ba)^{m+1}.
\notag
\end{align}
Using \eqref{AE} in \eqref{xaa} gives
\allowdisplaybreaks
\begin{align*}
za
&=
\bigg(s_{0, 0} \,+ \sum^m_{j = 1} s_{0, j}\bigg) a  - \sum^m_{j = 1}s_{0, j}\, a^2(ba)^{j-1}b
\\
&\quad
+  \sum^m_{r = 1}
\bigg( \sum^m_{j = 1} s_{0, j}\,(-1)^r \binom{j}{r} + s_{1,r-1}\,\bigg)a(ba)^r   +  s_{1, m}\, a(ba)^{m+1}
\\
&\quad
- \sum^m_{j = 0}s_{2, j}\, a^2(ba)^{j+1}b + \sum^m_{j=0}s_{2, j}\, a^2(ba)^j b.
\end{align*}
Changing index in the second and fourth sums and combining coefficients gives
\begin{align*}
za
&=
\Big(s_{0, 0}+ \sum^m_{j = 1} s_{0, j}\Big) a + \Big( {-}s_{0, 1} {+} s_{2, 0}\Big)a^2b
+
\sum^{m-1}_{j = 1} \Big( {-}s_{0, j+1} {-} s_{2, j-1} {+} s_{2, j} \Big) a^2(ba)^{j}b
\notag
\\
&\quad
 + \sum^m_{r = 1}\Big( \sum^m_{j = 1} s_{0, j}(-1)^r \binom{j}{r} + s_{1, r-1}\Big)a(ba)^r
+  s_{1, m}\, a(ba)^{m+1}
\\
&\quad
+\left(- s_{2, m-1} + s_{2, m}\right) a^2(ba)^m b - s_{2, m} \, a^2(ba)^{m +1}b.
\notag
\end{align*}
Comparing the coefficients in this expression with \eqref{ax}, we get this linear system:
\allowdisplaybreaks \begin{align}\label{*}
\begin{cases}
&\sum^m_{j = 1} s_{0, j} = 0,\\& -s_{0, 1} + s_{2, 0} - s_{1, 0} = 0,\\ &\notag
-s_{0, j+1} - s_{2, j-1}+s_{2, j}- s_{1, j} = 0\quad (1\leq j \leq m-1), \\&\tag{$\mathcal{T}$}
(-1)^r \sum^m_{j = 1}  \binom{j}{r}s_{0, j} + s_{1, r-1} - s_{0, r} = 0\quad (1\leq r  \leq m),\notag
\\
& s_{2, m} = s_{1, m}= s_{2, m-1} = 0.\notag\end{cases}
\end{align}
For $m < 2$, the only solution is trivial.
For $m \geq 2$, a calculation (details omitted) shows that \eqref{*} has $m{-}1$ linearly independent solutions.
For each $m$, we have the following solution:
\allowdisplaybreaks \begin{alignat*}{2}
\label{**}
s_{2,m-2} &= 1, \quad s_{2,j} = 0 \; (j\neq m-2),
&\quad
s_{0, j} &=  \frac{1}{3}(-1)^{j+1} \binom{m{-}1}{j{-}1} \; (1  \leq j \leq m-2),
\\
s_{0, m-1}
&=
\begin{cases}
(m{+}1)/3 &\text{if $m$ is even} \\
(3{-}m)/3  & \text{if $m$ is odd}
\end{cases},
&\quad
s_{0,m} &=
\begin{cases} -1 &\text{if $m$ is even} \\
-1/3 & \text{if $m$ is odd}
\end{cases},
\\
s_{1,i-1} &= -s_{0,i} \; (1\leq i\leq  m-2),
&\quad
s_{1, m-2} &= -s_{0, m-1}+1,
\\
s_{1, m-1}
&=
\begin{cases}
0  & \text{if  $m$ is even} \\
-2/3 & \text{if  $m$ is odd}
\end{cases}, &\quad s_{1,m} &=0.
\end{alignat*}
Using this solution for $z$, and observing that any solution for $m{-}1$ is also a solution for $m$,
we obtain a complete list of linearly independent solutions for \eqref{*}.
\end{proof}

\subsection{The alternating sum}

The structure constants for $A^\omega$ are $0$,
the set of ideal generators is empty, and hence
$U(A^\omega)$ is the free associative algebra on $a$ and $b$.
The Gelfand-Kirillov dimension of $U(A^\omega)$ is $\infty$.

\subsection{The cyclic sum }

The results are identical to those for the symmetric sum, since the structure constants are
  \[
  [e_2, e_1, e_1] = e_1,
  \qquad
  [e_1, e_2, e_2] = e_2,
  \qquad
  [e_1, e_1,e_1] =[e_2, e_2, e_2]= 0.
  \]

\subsection{The Lie family, $q = \infty$} \label{Lie}

The structure constants for $A^{\omega^{\tiny{\infty}}_{\tiny{L}}}$ are determined by
  \[
  [e_1, e_2, e_1 ] = 2e_1,
  \qquad
  [e_2, e_2, e_1] = -2 e_2.
  \]

\begin{lemma}\label{do}
The universal associative envelope $U(A^{\omega^{\tiny{\infty}}_{\tiny{L}}})$ is isomorphic to
the down-up algebra $A(2, -1, -2)$.
\end{lemma}

\begin{proof}
We have $U\big(A^{\omega^{\tiny{\infty}}_{\tiny{L}}}\big) = F\langle  a, b \rangle / I$,
where $I$ is the ideal generated by these two elements, which form a Gr\"obner basis:
$b^2a -2 bab + ab^2 +2 b$, $ba^2 -2 aba + a^2b +2 a$.
\end{proof}

\begin{remark}\label{sl}
If we replace ${\omega^{\tiny{\infty}}_{\tiny{L}}}$ by $\omega^{\prime} = [[-, - ], -]$ then
we get the 2-dimensional simple Lie triple system  $A^{\omega^{\prime}}$ with relations
$[e_1, e_2, e_1] = 2e_1$, $[e_1, e_2, e_2] = -2 e_2$;
the results for $U(A^{\omega^{\prime}})$ are identical to those for ${\omega^{\tiny{\infty}}_{\tiny{L}}}$.
\end{remark}

Benkart and Roby \cite{Roby} showed that the down-up algebra $A(2, -1, -2)$
is isomorphic to the universal associative envelope $U(\sltwo)$ of the simple Lie algebra of
$2 \times 2$ matrices of trace $0$ with basis $\{h, e, f \}$ and relations
$[e, f] = h$, $[h, e] = 2e$, and $[h, f]= -2f$.
In $U(\sltwo)$ we have
$ef - fe = h$, $he - eh= 2e$, $hf-fh= -2f$.

\begin{lemma} \label{helemma1}
If $\ell, k, m, j \ge 0$ then in  $U(\sltwo)$ we have
 \begin{align}
  e^\ell \cdot h^k
   &=
   \sum_{q = 0}^k
   (-1)^q 2^q \binom{k}{q} \ell^q
  h^{k-q} e^\ell,
  \label{one}
  \\
   h^k \cdot f^m
   &=
  \sum_{q=0}^k
  (-1)^q 2^q \binom{k}{q} m^q
   f^m h^{k-q},
   \label{two}
  \\
  e^\ell \cdot f^j
  &=
   \ell! j!\sum_{r=0}^{\min(j,\,\ell)}
  \frac{f^{j-r}}{(j - r)!}\binom{h{-}j{-}\ell{+}2r}{r}
  \frac{e^{\ell-r}}{(\ell -r)!}. \label{three}
  \end{align}
 \end{lemma}

\begin{proof}
For \eqref{one}, we use induction on $k$. The claim is clear for $k= 0$. To prove the claim for $k =1$, we use induction on $\ell$.
For $\ell = 1$, the claim holds since $eh  = he- 2e$.
Assume that $\ell \ge 1$. By the inductive hypothesis we have
    \allowdisplaybreaks \begin{align*}
   e^{\ell+1} h
  =
  e e^\ell h
  =
  \big(
  e h e^\ell - 2\ell e^{\ell +1}
  \big)
 =
    h e^{\ell +1} -2e^{\ell+1} - 2\ell e^{\ell+1},
    \end{align*}
so the claim is true for $k = 1$. For $k \ge 1$, the inductive hypothesis implies
    \allowdisplaybreaks \begin{align*}
  &e^\ell h^{k+1} =
  \sum_{q = 0}^k
  (-1)^q 2^q \binom{k}{q} \ell^q
  h^{k-q} e^\ell h
  =
  \sum_{q = 0}^k (-1)^q
  2^q \binom{k}{q} \ell^q
    h^{k - q} \big ( he^{\ell} - 2\ell e^\ell\big)
    \\
  &=
  \sum_{q=0}^k (-1)^q
  2^q \binom{k}{q} \ell^q
    h^{k-q+1}e^{\ell}
  +
  \sum_{q=1}^{k +1} (-1)^{q}
  2^{q} \binom{k}{q-1} \ell^{q}h^{k-q+1}e^{\ell}
  \\
  &=
  h^{k+1}e^{\ell}
  +
  \sum_{q=1}^k (-1)^q
  2^q
  \left[
  \binom{k}{q}
  {+}
  \binom{k}{q{-}1}
  \right]
  \ell^q
    h^{k+1-q} e^{\ell}
  + (-1)^{k +1}
  2^{k+1} \ell^{k+1}
  e^\ell.
  \end{align*}
Using Pascal's formula for binomial coefficients we obtain
  \allowdisplaybreaks \begin{align*}
  e^\ell h^{k+1}&=
 h^{k+1}  e^\ell
  +
 \sum_{q=1}^k (-1)^q
 2^q
 \binom{k{+}1}{q}
  \ell^q
   h^{k+1-q}  e^\ell
  +(-1)^{k + 1}
  2^{k+1} \ell^{k+1}
  e^\ell.
  \end{align*}
This proves \eqref{one}, and \eqref{two} is similar; for \eqref{three}, see Humphreys \cite[Lemma 26.2]{Humphreys}.
\end{proof}

\begin{theorem}\label{th}
The structure constants of $U(\sltwo)$ are
\begin{dmath*}
 ( f^i h^j  e^k )\cdot ( f^\ell h^m e^n )
 = \, k!\, \ell!   \sum_{r = 0}^{\min(\ell,\,k)}\!\!
  \sum_{q = 0}^j\sum^m_{i = 0} (-1)^{q + i} 2^{q+i} \binom{j}{q}\binom{m}{i}\frac{(\ell -r)^q}{(\ell - r)!}\frac{(k - r)^i}{(k - r)!}\times
    {f^{\ell-r+i}}
  h^{j-q}\binom{h - k -\ell + 2r}{r}
    h^{m-i} e^{k - r + n}.
\end{dmath*}
\end{theorem}

\begin{remark}
Using Theorem \ref{th} and the homomorphism $\psi\colon \sltwo \to  A(2, -1, 1)$ from \cite{Roby},
we obtain the structure constants of $U\big(A^{\omega^{\tiny{\infty}}_{\tiny{L}}}\big)$
with respect to the basis $\mathfrak{B}_2$ with $(c_1, c_2)= (-1,0)$ (see Lemma \ref{zhaolemma}).
\end{remark}

\subsection{The Lie family, $q = \frac12$ }

The structure constants for $A^{\omega^{\tiny{1/2}}_{\tiny{L}}}$ are zero.

\begin{lemma}\label{lj}
The universal associative envelope $U(A^{\omega^{\tiny{1/2}}_{\tiny{L}}})$ is isomorphic to the down-up algebra $A(0, 1, 0)$.
\end{lemma}

\begin{proof}
We have $U(A^{\omega^{\tiny{1/2}}_{\tiny{L}}}) = F\langle a, b\rangle / I$,
where $I$ is the ideal generated by these two elements, which form a Gr\"obner basis:
$b^2a - ab^2$, $ba^2 - a^2b$.
\end{proof}

\begin{remark}
If we replace ${\omega^{\tiny{1/2}}_{\tiny{L}}}$
by ${\omega^{\prime\prime}} = [- \circ-, -]$ then we get an anti-Lie triple system $A^{\omega^{\prime\prime}}$, and the results for $U(A^{\omega^{\prime\prime}})$ are the same as those for $A^{\omega^{\tiny{1/2}}_{\tiny{L}}}$.
\end{remark}

\begin{lemma}\label{rr}
If $i, j \ge 0$, then in $U(A^{\omega^{\tiny{1/2}}_{\tiny{L}}})$ we have
  \begin{align}
   b^i\cdot a^j
   &=\begin{cases}  a^{j-1} (ba) b^{i-1}   &\text{if $i, j$ are both odd,}
      \\
       a^jb^i   &\text{otherwise.} \label{ba}
       \end{cases}
  \\
  (ba)^j \cdot a^i
  &=
  \begin{cases}  a^{i+1} (ba)^{j-1} b   &\text{if $i$  is odd,  $j \neq 0$,}
  \\
   a^i(ba)^j   &\text{otherwise.}  \label{baa}
        \end{cases}
  \\
  b^i \cdot (ba)^j
  &=
  \begin{cases}  a (ba)^{j-1} b^{i+1}   &\text{if $i$ is odd, $j \neq 0$, }
  \\
   (ba)^j b^i  &\text{otherwise.} \end{cases}\label{bba}
   \\
   (ba)^i \cdot a(ba)^j &= \begin{cases}
    a^{2j+2}(ba)^{i-j-1}b^{2j+1} &\text{if $i > j$,}
    \\
    a^{2i+1}(ba)^{j-i}b^{2i} &\text{otherwise.}\label{haa}
    \end{cases}
   \end{align}
 \end{lemma}

\begin{proof}
For \eqref{ba}, we use induction on $i$. The claim is clear for $i = 0$. To prove the claim for $i = 1$ we use
induction on $j$. For $j = 0$ or $1$, the claim is obvious. For $j = 2$ the claim holds since $ba^2 = a^2 b $ .
We now prove the claim for $j\geq 2$. By the inductive hypothesis, we have
\allowdisplaybreaks \begin{align*}
ba^j &= \begin{cases} a^{j-2} baa  &\text{if $j{-}1$ is odd}
\\
a^{j-1} ba &\text{otherwise}
\end{cases}
=\begin{cases}  a^{j-1} ba   &\text{if $j$ is odd}
 \\
  a^{j-2}baa &\text{otherwise}  \end{cases}
= \begin{cases}
  a^{j-1} ba   &\text{if $j$ is odd}
   \\ a^jb  &\text{otherwise}
   \end{cases}.
\end{align*}
So the claim holds for $i = 1$.
We now consider the case $i\geq 1$.  If $i+1$ is odd, then the inductive hypothesis implies
\[ b^{i+1}a^j = bb^ia^j = b a^j b^i = \begin{cases} a^{j-1}(ba)b^i  &\text{if $j$ is odd}\\ a^{j}b^{i+1}     &\text{otherwise}         \end{cases}.\]
If $i+1$ is even, then the inductive hypothesis gives
\allowdisplaybreaks \begin{align*}
   b^{i+1}a^j &= bb^ia^j =\begin{cases} b a^{j-1}(ba)b^{i-1}    &\text{if $j$ is odd}
    \\
     ba^jb^i &\text{otherwise}\end{cases} =
\begin{cases} a^{j-1}b(ba)b^{i-1}  &\text{if $j$  is odd}  \\ a^j b^{i+1} &\text{otherwise}
 \end{cases}.
 \end{align*}
Using $b^2a = ab^2$ we get
$b^{i+1}a^j = a^{j-1}ab^2 b^{i-1}$ if $j$ is odd, $a^j b^{i+1}$ otherwise;
in both cases the result is $a^jb^{i+1}$.
This completes the proof of \eqref{ba}.

For \eqref{baa} we use induction on $i$.
The claim is obvious for $i = 0$. To prove the claim for $i = 1$, we use
induction on $j$. For $j = 0$, the claim is obvious. For $j = 1$, the claim holds by using $ba^2 = a^2b$. We now consider the case of general $j$. By the inductive hypothesis, we have
$(ba)^ja = ba (ba)^{j-1}a =  ba a^{2} (ba)^{j-2} b = a^2 ba(ba)^{j-2} = a^2(ba)^{j-1}b.$
So the claim is true for $i =1$.  We now consider the case $i\geq 1$. The claim is obvious for $j= 0$, so we assume that $j\neq 0$. If $i+1$ is odd,  then the inductive hypothesis implies
$ (ba)^ja^{i+1} = a^i(ba)^ja = a^{i+2}(ba)^{j-1}b$.
 If $i+1$ is even, then
 $
  (ba)^ja^{i+1}  = a^{i+1}(ba)^{j-1}ba
= a^{i+1}(ba)^j.
$
This completes the proof of \eqref{baa}.
The proof of \eqref{bba} follows by using the anti-automorphism $\eta$ from \cite{Roby} of the down-up algebra $A(0, 1, 0)$.

For \eqref{haa}, we use induction on $i$. The claim is obvious for $i = 0$. To prove the claim for $i = 1$, we use induction on $j$.
The claim holds for $j = 0$ by using $(ba) a = a^2b$.
For $j \geq 1$, $ba^2 = a^2b$ and \eqref{bba} imply
$
ba a (ba)^j = a^2b(ba)^j
 = a^3(ba)^{j-1}b^2.
$
So the claim is true for $i = 1$. We now consider the case of $i> 1$. By
the inductive hypothesis, we have
\allowdisplaybreaks \begin{align*}
(ba)^{i+1} a (ba)^j &= ba (ba)^ia(ba)^j = \begin{cases}
    ba a^{2j+2}(ba)^{i-j-1}b^{2j+1} &\text{if $i > j$ }
    \\
    ba a^{2i+1}(ba)^{j-i}b^{2i} &\text{if $i \leq j$}       \end{cases}.\notag
   \end{align*}
Therefore
   \begin{equation*}
 (ba)^{i+1} a (ba)^j = \begin{cases}
 b a^{2j+3}(ba)^{i-j-1}b^{2j+1} &\text{if $i > j$ }
    \\
    b a^{2i+2}(ba)^{j-i}b^{2i} &\text{if $i \leq j$}       \end{cases}.\label{i}
\end{equation*}
Two cases need to be considered.
(I) If $i+1 > j$, then $ i = j$ or $i >j $. Hence,
\allowdisplaybreaks \begin{align*}
 (ba)^{i+1} a (ba)^j = \begin{cases}
  b a^{2j+3}(ba)^{i-j-1}b^{2j+1} &\text{if $i > j$ }
 \\
 b a^{2i+2}b^{2i} &\text{if $i = j$}
    \end{cases}.
\end{align*}
Using \eqref{ba} we obtain
\allowdisplaybreaks \begin{align*}
 (ba)^{i+1} a (ba)^j = \begin{cases}
  a^{2j} ba^3(ba)^{i-j-1}b^{2j+1} = a^{2j+2} (ba)^{i - j}b^{2j+1} &\text{if $i > j$}
 \\
 a^{2i + 2}b^{2i +1} &\text{if $i = j$}
    \end{cases}.
\end{align*}
Therefore,
$
(ba)^{i+1} a (ba)^j = a^{2j+2}(ba)^{i-j}b^{2j+1}. $
(II)
If $ i +1 \leq j$, then $i < j$.  Hence
$
 (ba)^{i+1} a (ba)^j  = b a^{2i+2}(ba)^{j-i}b^{2i}.$
Using \eqref{ba} and \eqref{bba} we obtain
$
(ba)^{i+1} a (ba)^j = a^{2i+2}b(ba)^{j-i}b^{2i}
  = a^{2i+3}(ba)^{j-i-1}b^{2(i+1)}.
 $
Combining the results of (I) and (II) completes the proof of \eqref{haa}.
\end{proof}

\begin{theorem}\label{stanti}
The structure constants of $U(A^{\omega^{\tiny{1/2}}_{\tiny{L}}})$ are
\begin{dmath*}
 a^i (ba)^j b^k \cdot a^{\ell}(ba)^m b^n =
    \begin{cases}
    a^{i+{\ell}-1}(ba)^{j+m+1}b^{k-1+n} &\text{if $k, \, \ell$ are both odd,}
        \\
  a^{i+\ell}(ba)^{j+m} b^{k+n} &\text{if $k, \, \ell$ are both even,}
      \\
   \chi_{j, m}\, a^{2m + i+\ell+1} (ba)^{j-m-1} b^{2m+ k+n+1 }\\\quad +\, (1- \chi_{j, m})\, a^{2j+i+\ell}(ba)^{m-j}b^{2j+k+n} &\text{if $k$  is even, $\ell$ is odd,}
\\
\chi_{j, m-1}\,  a^{2m+i+\ell}(ba)^{j-m} b^{2m+k+n}\\ \quad + \, (1 -  \chi_{j, m-1})\,  a^{2j+i+\ell+1}(ba)^{m-j-1}b^{2j+k+n+1} & \text{if $k$  is odd, $\ell$ is even,}
\end{cases}
     \end{dmath*}
where $\chi_{\ell, t} = 1$ if $\ell > t$ and 0 otherwise.
       \end{theorem}

\begin{proof}
We use  equations \eqref{ba}, \eqref{baa} and \eqref{bba}.
If  $k$ and $\ell $ are odd, then
\begin{dmath*}
{a^i (ba)^j b^k \cdot a^{\ell} (ba)^m b^n =  a^i (ba)^ja^{{\ell}-1} (ba)b^{k-1}(ba)^mb^n
 = a^{i+{\ell}-1}(ba)^{j+1+m}b^{k-1+n}}.
\end{dmath*}
If $k$ and $\ell$ are even, then
\begin{dmath*}
{a^i (ba)^j b^k \cdot a^{\ell} (ba)^m b^n = a^i {(ba)^ja^{\ell}} b^k(ba)^mb^n
 = a^{i+\ell}(ba)^{j+m} b^{k+n}}.
\end{dmath*}
If  $k$ is even and $\ell$ is odd, then
\begin{align*}
&
a^i(ba)^j b^k \cdot a^{\ell} (ba)^mb^n
=
a^i (ba)^j a^{\ell} b^k (ba)^m b^n
\\
&=
\widehat{\delta}_{j,0} \, a^{i+ \ell +1} (ba)^{j-1}(b^{k+1}(ba)^m)b^n + \delta_{j,0}\, a^{i+\ell} b^k (ba)^mb^n
\\
&=
\widehat{\delta}_{j,0}\left[ \widehat{\delta}_{m,0}\, a^{i+\ell+1} {(ba)^{j-1}a (ba)^{m-1}} b^{k+2+n}
+\delta_{m,0}\, a^{i+\ell+1}(ba)^{j-1}b^{k+n+1}\right]
\\
&\qquad
+\delta_{j,0}\, a^{i+\ell} (ba)^mb^{k+n}.
\end{align*}
Using \eqref{haa} completes the proof.
If $k$ is odd and $\ell$ is even, then
\begin{dmath*}
 {a^i (ba)^j {b^k  a^{\ell}}(ba)^m b^n  = a^i (ba)^j a^{\ell} b^k(ba)^m b^n
 = a^{i+\ell}(ba)^j b^k (ba)^m b^n}
 = \widehat{\delta}_{m,0}\, a^{i+\ell} (ba)^{j} a (ba)^{m-1} b^{k+n+1} + \delta_{m, 0}\, a^{i+\ell}(ba)^j b^{k + n}.
\end{dmath*}
Using \eqref{haa} again completes the proof.
\end{proof}

\subsection{The anti-Jordan family, $q = \infty$}

The structure constants for $A^{\omega^{\tiny{\infty}}_{\tiny{AJ}}}$ are
  \[
  [e_1, e_1, e_2 ] = -2 e_1,
  \quad
  [e_2, e_1, e_1]= 2e_1,
  \quad
  [ e_1, e_2, e_2] = 2e_2,
  \quad
  [e_2, e_2, e_1] = -2e_2.
  \]

\begin{proposition}
The universal associative envelope $U(A^{\omega^{\tiny{\infty}}_{\tiny{AJ}}})$
is isomorphic to the down-up algebra $A(2, -1, -2)$, so we have
$U(A^{\omega^{\tiny{\infty}}_{\tiny{AJ}}}) \cong U(A^{\omega^{\tiny{\infty}}_{\tiny{L}}})$.
\end{proposition}
\begin{proof}
Similar to the proof of Lemma \ref{do}.
\end{proof}

\subsection{The anti-Jordan family, $q = \frac12$}

The structure constants for $A^{\omega^{\tiny{1/2}}_{\tiny{AJ} }} $ are zero.

\begin{proposition}
The universal associative envelope $ U\big(A^{\omega^{\tiny{1/2}}_{\tiny{AJ} }}\big)$
is isomorphic to the down-up algebra $A(0, 1, 0)$, so we have
$U\big(A^{\omega^{\tiny{1/2}}_{\tiny{AJ} }}\big) \cong U\big(A^{\omega^{\tiny{1/2}}_{\tiny{L} }}\big)$.
\end{proposition}

\begin{proof}
Similar to the proof of Lemma \ref{lj}.
\end{proof}
\begin{remark}
The structure $A^{\omega^{\tiny{1/2}}_{\tiny{AJ} }} $ is an example of non-simple anti-Jordan triple systems. For the classification of simple finite-dimensional anti-Jordan triple systems see \cite[Theorem 6]{simen}. For the universal associative envelope of the anti-Jordan triple system of all $n \times n$ matrices see \cite[Theorem 6.2.12]{Elgendy}.
\end{remark}


\section{Finite dimensional envelopes} \label{finite}

In this section, we consider the trilinear operations of ``Jordan type''.

\subsection{The Jordan family, $q = \infty$} \label{Jordan}

The structure constants for $A^{\omega^{\tiny{\infty}}_{\tiny{J}}}$ are
  \[
  [e_1, e_2, e_1] = 2 e_1,
  \qquad
  [ e_2, e_1, e_2 ] = 2 e_2.
  \]

\begin{theorem}\label{Jinf}
A basis for $U(A^{\omega^{\tiny{\infty}}_{\tiny{J}}})$ consists of the elements 1, $a$, $b$, $ab$, $ba$.
The structure constants are
$a \cdot b = ab$, $a \cdot ba = a$, $b \cdot a = ba$, $b \cdot ab = b$, $ab \cdot a = a$,
$ab \cdot ab = ab$, $ba \cdot b = b$, $ba \cdot ba = ba$.
The Wedderburn decomposition is $U(A^{\omega^{\tiny{\infty}}_{\tiny{J}}}) = \rational \oplus M_{2\times 2}$.
The only finite dimensional irreducible representations are
the trivial 1-dimensional representation and the natural 2-dimensional representation.
\end{theorem}

\begin{proof}
We have $U = U(A^{\omega^{\tiny{\infty}}_{\tiny{J}}}) = F\langle a, b \rangle / I$
where $I$ is generated by $b^3$, $b^2a + ab^2$, $bab - b$, $ba^2 + a^2b$, $aba - a$, $a^3$.
We compute a Gr\"obner basis of $I$.
There are four compositions with normal forms
$ab^2$, $a^2b$, $b^2$, $a^2$.
Including these with the original generators and self-reducing the resulting set produces the four generators
$bab - b$, $aba - a$, $b^2$, $a^2$.
All compositions of these elements reduce to $0$, and so we have a Gr\"obner basis.
A basis for the quotient algebra consists of the cosets of the monomials which are not divisible by
the leading monomial of any element of the Gr\"obner basis.
This gives the stated basis for $U$.
It follows that $U$ satisfies $a^2 = 0$, $b^2 = 0$, $aba = a$, $bab = b$
and these give the stated structure constants.

To decompose $U$ we follow \cite{Bremner}.
Using \cite[Corollary 12]{Bremner} we verify that the radical is zero,
and hence $U$ is semisimple.
By \cite[Corollary 15]{Bremner} the center $Z(U)$
has dimension 2, basis $z_1 = 1$, $z_2 = ab + ba$, and structure constants
$z_1 \cdot z_1 = z_1$, $z_1 \cdot z_2 = z_2$, $z_2 \cdot z_2 = z_2$.
Since $z^2_2 = z_2$, the minimal polynomial of $z_2$ is $t^2 - t$.
Thus $Z(U)$ splits in two 1-dimensional ideals with bases $z_2 - z_1$ and $z_2$.
Scaling these basis elements to obtain idempotents gives $e_1 = -z_2 +z_1$, $e_2 = z_2$.
The corresponding elements in $U$ are
$e_1 = -ab-ba+1$, $e_2 = ab+ba$.
The ideals in $U$ generated by $e_1$ and $e_2$ have dimensions 1 and 4
respectively, and this gives the Wedderburn decomposition.
\end{proof}

\subsection{The Jordan family, $q = 0$}

The structure constants for $A^{\omega^{\tiny{0}}_{\tiny{J}}}$ are
  \[
  [e_1, e_2, e_1] = [e_2, e_1, e_1] = e_1,
  \qquad
  [ e_2, e_1, e_2] = [e_1, e_2, e_2]= e_2.
  \]

\begin{theorem}\label{jq0}
A basis for $U(A^{\omega^{\tiny{0}}_{\tiny{J}}})$ consists of the elements
$1$, $a$, $b$, $a^2$, $ab$, $ba$, $b^2$, $aba$, $ab^2$.
The structure constants are
$a\cdot a = a^2$,
$a  \cdot b = ab$,
$a  \cdot ba = aba$,
$a \cdot b^2 = ab^2$,
$b \cdot a = ba$,
$b \cdot b = b^2$,
$b \cdot a^2 = a-aba$,
$b \cdot ab = b-ab^2$,
$b \cdot aba = ba$,
$b \cdot ab^2 = b^2$,
$ab \cdot a = aba$,
$ab \cdot b = ab^2$,
$ab \cdot a^2 = a^2$,
$ab \cdot ab = ab$,
$ab \cdot aba = aba$,
$ab \cdot ab^2 = ab^2$,
$ba \cdot a = a-aba$,
$ba \cdot b = b-ab^2$,
$ba \cdot ba = ba$,
$ba \cdot b^2 = b^2$,
$aba \cdot a = a^2$,
$aba \cdot b = ab$,
$aba \cdot ba = aba$,
$aba \cdot b^2 = ab^2$.
The Wedderburn decomposition is
$U(A^{\omega^{\tiny{0}}_{\tiny{J}}}) = \mathfrak{R} \oplus \rational \oplus M_{2\times 2}$
where $\mathfrak{R}$ is the radical of dimension 4.
There are only two finite dimensional irreducible representations.
\end{theorem}

\begin{proof}
We have $U = U(A^{\omega^{\tiny{0}}_{\tiny{J}}}) = F\langle a, b \rangle / I$ where $I$ is generated by
$b^3$,
$b^2a$,
$bab + ab^2 - b$,
$ba^2 + aba - a$,
$a^2b$,
$a^3$.
This set is a Gr\"obner basis for $I$.
Hence $U$ is finite dimensional and has the stated basis.
The following relations hold in $U$:
$b^3 = 0$,
$b^2a = 0$,
$bab = -ab^2 + b$,
$ba^2=- aba + a$,
$a^2b = 0$,
$a^3=0.$
These imply the stated structure constants.
Using \cite[Corollary 12]{Bremner}, a basis of the radical $\mathfrak{R} = \mathfrak{R}(U)$ consists of the elements
$\xi_1 = a - aba,\,\, \xi_2 = a^2,\,\, \xi_3 = b^2,\,\, \xi_4= ab^2.$
Hence we have these relations in $Q = U/\mathfrak{R}$:
$a = aba,\,\, a^2 = b^2 = ab^2 = 0.$
The semisimple quotient $Q$ has dimension $5$, and a basis consists of the cosets of
$\eta_1 = 1$, $\eta_2 = b$, $\eta_3 = ab$, $\eta_4 = ba$, $\eta_5= aba$.
The center $Z(Q)$ has dimension $2$, basis
$z_1=\eta_1,\,\,  z_2=\eta_3 +\eta_4,$ and structure constants
$z_1\cdot z_1= z_1$, $z_1\cdot z_2 = z_2 \cdot z_1 = z_2$, $z_2 \cdot z_2 =  z_2$.
Since $ z^2_2 = z_2$, the minimal polynomial of $z_2$ is $t^2 - t$. Thus $ Z(Q) = J \oplus K$ where
$J= \langle z_2-z_1 \rangle$  and $K= \langle z_2\rangle$
and both ideals are 1-dimensional.
Scaling the basis elements to obtain idempotents gives $e_1 = z_1 - z_2$, $e_2 = z_2$.
The corresponding elements in $Q$ are
$e_1 = \eta_1 - \eta_3 - \eta_4, \,\, e_2 = \eta_3 +\eta_4.$
The ideals in $Q$ generated by $e_1$ and $e_2$ have dimensions 1 and 4 respectively,
and this gives the Wedderburn decomposition.
\end{proof}

\subsection{The Jordan family, $q = \frac12$}

The structure constants for $A^{\omega^{\tiny{1/2}}_{\tiny{J}}}$ are
\allowdisplaybreaks \begin{align*}
 &[e_1, e_1, e_2] = [e_1, e_2, e_1]=  [e_2, e_1, e_1]= 2 e_1,
  \\& [e_2, e_2, e_1] = [ e_2, e_1, e_2] = [e_1, e_2, e_2]= 2e_2.
\end{align*}

\begin{theorem}
We have the isomorphism
$U(A^{\omega^{\tiny{1/2}}_{\tiny{J}}}) \cong  U(A^{\omega^{\tiny{\infty}}_{\tiny{J}}})$.
\end{theorem}

\begin{proof}
We have $U = U(A^{\omega^{\tiny{1/2}}_{\tiny{J}}}) = F\langle a, b\rangle / I$ where $I$ is generated by
$b^3$,
$b^2a + \tfrac12 bab - \tfrac12 b$,
$b^2a +2 bab +3 ab^2 -2 b$,
$ba^2 + \tfrac23 aba + \tfrac13 a^2b - \tfrac23 a$,
$aba +2 a^2b - a$,
$a^3$.
The first iteration of the Gr\"obner basis algorithm produces the seven compositions
$bab^2 - 2 b^2$,
$bab^2 - \tfrac12 b^2$,
$bab - b$,
$bab +2 ab^2 - b$,
$bab + \tfrac{3}{2} ab^2 -b$,
$a^2b$,
$a^2$.
Including these with the original generators, and self-reducing the resulting set,
produces the same ideal generators as for $q = \infty$;
hence the two quotient algebras are isomorphic.
\end{proof}

\subsection{The Jordan family, $q = 1$}

The structure constants for $A^{\omega^{\tiny{1}}_{\tiny{J}}}$ are
  \[
  [e_1, e_1, e_2] = [e_1, e_2, e_1]= e_1,
  \qquad
  [e_2, e_2, e_1] = [e_2, e_1, e_2]= e_2.
  \]

\begin{theorem}\label{jor1}
A basis for $U(A^{\omega^{\tiny{1}}_{\tiny{J}}})$ consists of the elements
$1$, $a$, $b$, $a^2$, $ab$, $ba$, $b^2$, $a^2b$, $bab$.
The structure constants of $U(A^{\omega^{\tiny{1}}_{\tiny{J}}})$ are
$a \cdot a = a^2$,
$a\cdot b = ab$,
$a\cdot ab = a^2b$,
$a\cdot ba = a-a^2b$,
$a\cdot bab = ab$,
$b \cdot a = ba$,
$b\cdot b =b^2$,
$b\cdot ab = bab$,
$b\cdot ba = b-bab$,
$b\cdot bab = b^2$,
$a^2\cdot b = a^2b$,
$a^2\cdot ba = a^2$,
$a^2\cdot bab = a^2b$,
$ab\cdot a = a - a^2b$,
$ab\cdot ab = ab$,
$ba\cdot b=bab$,
$ba\cdot ba = ba$,
$ba\cdot bab=bab$,
$b^2\cdot a = b - bab$,
$b^2\cdot ab = b^2$,
$a^2b\cdot a = a^2$,
$a^2b\cdot ab = a^2b$,
$bab\cdot a = ba$,
$bab\cdot ab = bab$.
The Wedderburn decomposition is
$U(A^{\omega^{\tiny{1}}_{\tiny{J}}}) = \mathfrak{R} \oplus \rational \oplus M_{2\times 2}$
where $\mathfrak{R}$ is the radical of dimension 4.
There are two finite dimensional irreducible representations.
\end{theorem}

\begin{proof}
The original set of generators of the ideal $I$ is a Gr\"obner basis and consists of the six elements
$b^3$,
$b^2a + bab - b$,
$ba^2$,
$ab^2$,
$aba + a^2b - a$,
$a^3$.
Hence  $U = U(A^{\omega^{\tiny{1}}_{\tiny{J}}})$ is finite dimensional with the stated basis.
The following relations hold in $U$:
$b^3 = 0$,
$b^2a =- bab + b$,
$ba^2 = 0$,
$ab^2 = 0$,
$aba =- a^2b + a$,
$a^3 = 0$.
These give the stated structure constants.
A basis of the radical $\mathfrak{R} = \mathfrak{R}(U)$ consists of the elements
$\xi_1 = b - bab$, $\xi_2 = a^2$, $\xi_3 = b^2$, $\xi_4= a^2b$
which give these relations in $Q = U/\mathfrak{R}$:
$ b = bab,\, a^2 = b^2 = a^2b = 0.$
The semisimple quotient $Q$ has dimension $5$ and a basis consists of the cosets of
$\eta_1= 1$, $\eta_2= a$, $\eta_3= ab$, $\eta_4= ba$, $\eta_5 = bab$.
The center $Z(Q)$ has dimension 2 with basis
$z_1=\eta_1$, $z_2=\eta_3 +\eta_4$ and structure constants
$z_1\cdot z_1= z_1,\, z_1\cdot z_2 = z_2 \cdot z_1 = z_2,\,z_2 \cdot z_2 =  z_2.$
Since $ z^2_2 = z_2$, the minimal polynomial of $z_2$ is $t^2 - t$.
Thus $ Z(Q) = J \oplus K$, where $J= \langle z_2-z_1 \rangle$ and $K = \langle z_2\rangle$;
both ideals are 1-dimensional.
Scaling these basis elements to obtain idempotents gives
$e_1 = z_1 - z_2$, $e_2 = z_2$.
The corresponding elements of $Q$ are
$e_1 = \eta_1 - \eta_3 - \eta_4$, $e_2 = \eta_3 +\eta_4$.
The ideals in $Q$ generated by $e_1$ and $e_2$ have dimensions 1 and 4 respectively,
and this gives the Wedderburn decomposition.
\end{proof}

\subsection{The anti-Jordan family, $q = -1$} \label{anti-Jordan}

The structure constants for $A^{ \omega^{\tiny{-1}}_{\tiny{AJ}} }$ are
  \[
  [e_1, e_1, e_2] = -e_1,
  \qquad
  [e_1, e_2, e_1 ] = e_1,
  \qquad
  [e_2, e_1, e_2]= e_2,
  \qquad
  [e_2, e_2, e_1] = -e_2.
  \]

\begin{theorem}\label{AJ-1}
We have the isomorphisms
$U\big(A^{ \omega^{\tiny{-1}}_{\tiny{AJ}} }\big)
\cong
U\big(A^{ \omega^{\tiny{1/2}}_{\tiny{J}} }\big)
\cong
U\big(A^{ \omega^{\tiny{\infty}}_{\tiny{J}} }\big)$.
\end{theorem}

\begin{proof}
We have $U = U\big(A^{ \omega^{\tiny{-1}}_{\tiny{AJ}} }\big) = F\langle a, b\rangle / I$
where $I$ is generated by
$b^2a - bab + b$ and
$aba - a^2b - a$.
The first iteration of the Gr\"obner basis algorithm produces one composition, $bab - b$.
Including this element with the original generators, and self-reducing the resulting set,
produces a new set of three generators:
$b^2a$,
$bab - b$,
$aba - a^2b - a$.
The second iteration produces three compositions:
$ba^2b$,
$a^2b^2$,
$b^2$.
Including these elements with the previous generators, and self-reducing the resulting set,
produces a new set of four generators:
$ba^2b$,
$bab - b$,
$aba - a^2b - a$,
$b^2$.
The third iteration produces two compositions:
$ba^3b + ba^2$,
$a^2b$.
Including these elements with the previous generators, and self-reducing the resulting set,
produces a new set of five generators:
$bab - b$,
$ba^2$,
$aba - a$,
$a^2b$,
$b^2$.
The fourth iteration produces one composition, $a^2$.
Including this element with the previous generators, and self-reducing the resulting set,
produces a new set of four generators:
$bab - b$,
$aba - a$,
$b^2$,
$a^2$.
This is a Gr\"obner basis for the ideal, and is the same Gr\"obner basis as for the Jordan cases
$q = \infty$, $q = \frac12$;
hence the quotient algebras are isomorphic.
\end{proof}

\subsection{The anti-Jordan family, $q = 2$}

The structure constants for $A^{\omega^{\tiny{2}}_{\tiny{AJ} }}$ are
  \[
  [e_1, e_2, e_1]= e_1,
  \qquad
  [e_2, e_1, e_1]= -e_1,
  \qquad
  [e_2, e_1, e_2]= e_2,
  \qquad
  [e_1, e_2, e_2] = -e_2.
  \]

\begin{proposition}
We have
$U\big(A^{ \omega^{\tiny{2}}_{\tiny{AJ}} }\big)\cong    U(A^{ \omega^{\tiny{-1}}_{\tiny{AJ}} })$.
\end{proposition}
\begin{proof}
Similar to the proof of Theorem \ref{AJ-1}.
\end{proof}

\subsection{The last nine operations} \label{fourth}

We first consider the fourth family with $q = \infty$.
The structure constants for $A^{\omega^{\tiny{\infty}}_{\tiny{F}}}$ are
  \begin{align*}
  &
  [e_1, e_1, e_2] = [e_2, e_1, e_1]= -e_1,
  \qquad
  [e_1, e_2, e_1] = e_1,
  \\
  &
  [e_2, e_2, e_1] = [e_1, e_2, e_2]= -e_2,
  \qquad
  [e_2, e_1, e_2] = e_2.
  \end{align*}

\begin{proposition} \label{fourthinfinity}
We have
$U(A^{\omega^{\tiny{\infty}}_{\tiny{F} }}) \cong   U(A^{\omega^{\tiny{\infty}}_{\tiny{J}}})$.
\end{proposition}

\begin{proof}
We have
$U\big(A^{ \omega^{\infty}_{\tiny{F}} }\big) = F\langle a, b \rangle / J$
where $J$ is generated by $b^3$, $b^2a - bab + ab^2 + b$, $bab-b$, $ba^2-aba+a^2b+a$, $aba-a$, $a^3$.
Self-reducing this set of generators gives the set of generators for the Jordan case, $q =\infty$
(see the proof of Theorem \ref{Jinf}).
\end{proof}

For the fourth family with $q = 0$, the structure constants for $A^{\omega^{\tiny{0}}_{\tiny{F} }}$ are
  \[
  [e_1, e_2, e_1] = e_1, \qquad [ e_2, e_1, e_2]= e_2.
  \]

\begin{proposition}
We have
$U( A^{\omega^{\tiny{0}}_{\tiny{F} }} )\cong U( A^{\omega^{\tiny{0}}_{\tiny{J} }})$.
\end{proposition}

\begin{proof}
Similar to the proof of Proposition \ref{fourthinfinity}.
\end{proof}

For the fourth family with $q = 1$, the structure constants for $A^{\omega^{\tiny{1}}_{\tiny{F} }}$ are
  \[
  [e_1, e_2, e_1] = e_1,
  \qquad
  [e_2, e_1, e_2] = e_2.
  \]

\begin{proposition}
We have
$U( A^{\omega^{\tiny{1}}_{\tiny{F} }} ) \cong U( A^{\omega^{\tiny{1}}_{\tiny{J} }})$.
\end{proposition}

\begin{proof}
Similar to the proof of Proposition \ref{fourthinfinity}.
\end{proof}

We consider the last six operations together.

\smallskip

\noindent
Fourth family, $q = -1$:
The structure constants for $A^{\omega^{\tiny{-1}}_{\tiny{F} }}$ are
  \[
  [e_1, e_2, e_1]= e_1,
  \qquad
  [e_2, e_1, e_1] = 2e_1,
  \qquad
  [e_2, e_1, e_2]= e_2,
  \qquad
  [e_1, e_2, e_2]= 2e_2.
  \]
Fourth family, $q = 2$:
The structure constants for $A^{\omega^{\tiny{2}}_{\tiny{F} }}$ are
  \[
  [e_1, e_1, e_2] = 2e_1,
  \qquad
  [e_1, e_2, e_1] = e_1,
  \qquad
  [e_2, e_2, e_1]= 2e_2,
  \qquad
  [e_2, e_1, e_2]=e_2.
  \]
Fourth family, $q = \tfrac{1}{2}$:
The structure constants for $A^{\omega^{\tiny{1/2}}_{\tiny{F} }}$ are
  \begin{align*}
  &
  [e_1, e_1, e_2] = [e_1, e_2, e_1] = [e_2, e_1, e_1] = e_1,
  \\
  &
  [e_2, e_2, e_1]=[e_2, e_1, e_2] = [e_1, e_2,   e_2] = e_2.
  \end{align*}
Cyclic commutator:
The structure constants for $A^{\omega_{\tiny{cc} }}$ are
  \[
  [e_1, e_1, e_2] = -e_1,
  \qquad
  [e_1, e_2, e_1]= e_1,
  \qquad
  [e_2,e_2, e_1] = -e_2,
  \qquad
  [e_2, e_1, e_2]= e_2.
  \]
Weakly commutative operation:
The structure constants for $A^{\omega_{wc}}$ are
  \begin{align*}
  &
  [e_1, e_1, e_2] = -e_1,
  \qquad
  [e_1, e_2, e_1] = e_1,
  \qquad
  [e_2, e_1, e_1] = 2e_1,
  \\
  &
  [e_2, e_2, e_1] = -e_2,
  \qquad
  [e_2, e_1, e_2] = e_2,
  \qquad
  [e_1, e_2, e_2] = 2e_2.
  \end{align*}
Weakly anti-commutative operation:
The structure constants for $A^{\omega_{wa}}$ are
  \begin{align*}
  &
  [e_1, e_2, e_1] = [e_1, e_1, e_2]= e_1,
  \qquad
  [e_2, e_1, e_1] = -2e_1,
  \\
  &
  [e_2, e_1, e_2] = [e_2, e_2, e_1] = e_2,
  \qquad
  [e_1, e_2, e_2] = -2e_2.
  \end{align*}

\begin{proposition}\label{F-1}
We have the following isomorphisms:
\[
U(A^{\omega^{\tiny{-1}}_{\tiny{F} }})
\cong
U(A^{\omega^{\tiny{2}}_{\tiny{F} }})
\cong
U(A^{\omega^{\tiny{1/2}}_{\tiny{F} }})
\cong
U(A^{\omega_{\tiny{cc} }})
\cong
U(A^{\omega_{\tiny{wc} }})
\cong
U(A^{\omega_{wa}})
\cong
U(A^{\omega^{\tiny{\infty}}_{\tiny{J} }}).
\]
\end{proposition}

\begin{proof}
We have $U(A^{\omega^{\tiny{-1}}_{\tiny{F} }}) = F\langle a, b \rangle / I$ where $I$ is generated by
$b^3$, $b^2a+bab+ab^2-b$, $b^2a+\tfrac{1}{2}ab^2$, $bab+\tfrac{1}{2}ab^2-b$,
$ba^2+2aba-2a$, $ba^2+aba+a^2b-a$, $ba^2+2a^2b$, $a^3$.
We compute a Gr\"obner basis for $I$.
The first iteration produces eight compositions with the normal forms
$a^2b^2$, $a^2ba-a^2$, $a^2ba$, $aba-a^2b-a$, $ab^2$, $a^2b$, $b^2$, $a^2$.
Including these elements with the original generators, and then self-reducing the resulting set,
produces a new set of four ideal generators:
$bab-b$, $aba-a$, $a^2$, $b^2$.
This is a Gr\"obner basis $I$. In fact, this is the same Gr\"obner basis as in the Jordan case, $q = \infty$
(see the proof of Theorem \ref{Jinf}).
The other cases are similar.
\end{proof}


\section*{Acknowledgements}
The author thanks the anonymous referee for helpful comments. This work forms part of the author's doctoral thesis.
The author was supported by a Teacher Scholar Doctoral Fellowship from the University of Saskatchewan.
She thanks her supervisor, Prof.~Murray Bremner, for many suggestions and insights
which greatly improved this paper, and Prof.~Chris Soteros for the reference to \cite[Example 10.3]{vanLintWilson}.


\end{document}